\documentclass[reqno]{amsart}

\usepackage{amsmath, amssymb, amsthm}
\usepackage{dsfont}
\usepackage{cleveref}
\usepackage{color}
\usepackage{natbib}
\usepackage{graphicx}
\usepackage{booktabs}
\usepackage{algorithm}
\usepackage{algorithmic}
\usepackage{mathtools}
\usepackage{threeparttable}
\usepackage{enumitem}
\usepackage{float}
\newtheorem{theorem}{Theorem}[section]
\newtheorem{lemma}[theorem]{Lemma}
\newtheorem{corollary}[theorem]{Corollary}
\newtheorem{proposition}[theorem]{Proposition}

\newtheorem{assumption}{Assumption}
\newtheorem{definition}{Definition}

\newtheorem{remark}{Remark}
\theoremstyle{definition}

\theoremstyle{remark}

\newcommand{\bfv}{\textbf{v}}

\newcommand{\E}{\mathbb{E}}
\newcommand{\R}{\mathbb{R}}
\renewcommand{\P}{\mathbb{P}}
\newcommand{\Pt}{P^{(t)}}

\newcommand{\V}{\mathcal{V}}
\renewcommand{\H}{\mathcal{H}}
\newcommand{\HK}{\mathcal{H}_K}

\newcommand{\HKt}{\mathcal{H}_K^{(t)}}

\newcommand{\Kt}{K^{(t)}}
\newcommand{\kt}{k^{(t)}}

\newcommand{\M}{\mathcal{M}}
\newcommand{\bM}{\mathcal{\bar{M}}}
\renewcommand{\O}{\mathcal{O}}
\newcommand{\X}{\mathcal{X}}
\newcommand{\A}{\mathcal{A}}
\newcommand{\Y}{\mathcal{Y}}

\newcommand{\one}{\mathds{1}}
\newcommand{\valpha}{\boldsymbol{\alpha}}

\renewcommand{\hat}{\widehat}

\newcommand\blfootnote[1]{%
  \begingroup
  \renewcommand\thefootnote{}%
  \footnote{#1}%
  \endgroup
}

\title[KDPE based on the Universal Least Favorable Submodel]{Kernel Debiased Plug-in Estimation based on the Universal Least Favorable Submodel}
\begin{document}
\maketitle

% Custom authors + affiliations block 
\begin{center}
  \normalsize
  Haiyi Chen$^{1,*}$, Yang Liu$^{2,*}$, Ivana Malenica$^{1,**}$\\[4pt]\blfootnote{$^{*}$First Author}\blfootnote{$^{**}$Corresponding author: \texttt{imalenic@unc.edu}}
  \small
  $^{1}$Department of Biostatistics, UNC at Chapel Hill, USA\\
  $^{2}$Department of Statistics \& Operations Research, UNC at Chapel Hill, USA
  \\[2pt]
\end{center}

\begin{abstract}
We propose ULFS-KDPE, a kernel debiased plug-in estimator based on the universal least favorable submodel, for estimating pathwise differentiable parameters in nonparametric models. The method constructs a data-adaptive debiasing flow in a reproducing kernel Hilbert space (RKHS), producing a plug-in estimator that achieves semiparametric efficiency without requiring explicit derivation or evaluation of efficient influence functions. We place ULFS-KDPE on a rigorous functional-analytic foundation by formulating the universal least favorable update as a nonlinear ordinary differential equation on probability densities. We establish existence, uniqueness, stability, and finite-time convergence of the empirical score along the induced flow. Under standard regularity conditions, the resulting estimator is regular, asymptotically linear, and attains the semiparametric efficiency bound simultaneously for a broad class of pathwise differentiable parameters. The method admits a computationally tractable implementation based on finite-dimensional kernel representations and principled stopping criteria. In finite samples, the combination of solving a rich collection of score equations with RKHS-based smoothing and avoidance of direct influence-function evaluation leads to improved numerical stability. Simulation studies illustrate the method and support the theoretical results.
\end{abstract}

\section{Introduction}

Semiparametric efficiency theory provides a principled framework for constructing estimators that achieve optimal asymptotic variance in rich, nonparametric statistical models. Central to this theory is the notion of {pathwise differentiability} and its associated canonical gradient, or efficient influence function  (EIF), which characterizes both the first-order behavior of regular estimators and the semiparametric efficiency bound \cite{bickel1998, van2000asymptotic, tsiatis2006}. Canonical estimators widely used in causal inference, such as one-step estimators and targeted maximum likelihood estimators (TMLE), exploit this structure by ensuring that the estimator admits an asymptotic linear expansion with influence function (IF) equal to the EIF \cite{bickel1998, vanderLaanRubin2006TMLE, book2011, book2018}. These approaches have been highly successful in causal inference and missing-data problems, but they typically require explicit analytical knowledge of the EIF and are often tailored to a single target parameter \cite{book2011, book2018}.

TMLE is particularly prominent in this literature. It constructs a {plug-in estimator} by updating an initial estimate of the data-generating distribution along a locally least favorable parametric submodel (LLFS) until the EIF estimating equation is approximately solved \cite{vanderLaanRubin2006TMLE}. Under suitable regularity conditions, the resulting estimator respects the statistical model and achieves semiparametric efficiency. More recently, Cho et al.\ proposed the kernel debiased plug-in estimator (KDPE), which approaches efficiency from a different perspective \cite{cho2024kernel}. KDPE embeds the debiasing problem in a reproducing kernel Hilbert space (RKHS) and constructs data-adaptive fluctuations that approximately solve empirical score equations without requiring explicit expressions for the EIF. By leveraging RKHS geometry and finite-dimensional kernel representations, KDPE provides a flexible and computationally tractable route to efficient estimation in complex models.

Both canonical TMLE and KDPE, however, rely on \emph{local} notions of least favorability; updates are constructed to be optimal only infinitesimally at the current distribution. This motivates consideration of distributional paths that enforce least favorability \emph{globally}. The concept of a \emph{universal least favorable submodel} (ULFS) defines a distributional path whose score coincides with the EIF at every point along the path, rather than only at its origin \cite{one_step_tmle}. Geometrically, such a path follows the direction of maximal change in the target parameter per unit of information uniformly along its trajectory. Targeting along a universal least favorable path can therefore solve the EIF estimating equation in a single step while avoiding unnecessary likelihood fluctuation. Still, derivation and evaluation of the EIF is needed \cite{one_step_tmle}.

In this paper, we unify these ideas by proposing a new estimator—\emph{ULFS--KDPE}, also referred to as the \emph{one-step KDPE}—that combines the global optimality of universal least favorable paths with the computational and theoretical strengths of RKHS-based debiasing. We construct a kernel-restricted surrogate of the universal least favorable path, defined as the solution to a nonlinear ordinary differential equation on densities whose velocity field is given by an RKHS-valued Riesz representer of empirical moment deviations. Crucially, this construction does not require explicit knowledge of the EIF for any target parameter. Instead, it produces a single data-adaptive flow that simultaneously debiases all pathwise differentiable parameters whose canonical gradients lie in the $L^2(P_0)$-closure of the RKHS.

The resulting estimator is a {plug-in estimator} obtained by evaluating the target parameter at the final distribution along the universal least favorable flow. By construction, the flow increases the empirical log-likelihood monotonically and terminates once the relevant empirical score equations are approximately satisfied, thereby avoiding convergence pathologies associated with iterative local targeting. The global least favorability of the path ensures that the desired bias reduction is achieved with minimal likelihood fluctuation, leading to improved finite-sample stability relative to both iterative TMLE and locally targeted KDPE, particularly in settings with limited overlap.

From a theoretical perspective, we place ULFS--KDPE on a rigorous functional-analytic foundation. We formulate the universal least favorable update as a density-valued ordinary differential equation and establish existence, uniqueness, and stability of its solutions in appropriate Hölder spaces, along with preservation of normalization and positivity and finite-time convergence of the algorithm. Statistically, we show that the resulting plug-in estimator is regular, asymptotically linear, and semiparametrically efficient for all pathwise differentiable target parameters satisfying standard remainder conditions. Notably, this efficiency is achieved {simultaneously} across parameters without modifying the algorithm or specifying parameter-specific EIFs.
%Moreover, by solving a rich collection of score equations in a universal RKHS, the estimator can potentially capture structure beyond first-order debiasing, yielding further finite-sample advantages \cite{pimentel2025scorepreservingtargetedmaximumlikelihood}.

Finally, we demonstrate that ULFS--KDPE is computationally practical. Although the underlying construction is infinite-dimensional, each update admits a finite-dimensional representation involving only kernel evaluations at the observed data points. The resulting algorithm resembles a stabilized gradient flow in RKHS, comes equipped with principled stopping criteria tied directly to empirical score equations, and exhibits favorable finite-sample behavior in simulation studies, including in challenging regimes where iterative targeting methods exhibit instability.

\subsection{Contributions}

This paper makes the following contributions. First, we introduce \emph{ULFS--KDPE}, a kernel-based one-step debiased plug-in estimator that realizes a universal least favorable path within a reproducing kernel Hilbert space. The method produces a single data-adaptive distributional flow that simultaneously debiases a broad class of pathwise differentiable target parameters without requiring explicit efficient influence functions. Second, we formulate the universal least favorable update as a nonlinear ordinary differential equation on probability densities. We establish existence, uniqueness, stability, and finite-time convergence of its solutions in appropriate Hölder spaces, thereby providing a rigorous functional-analytic foundation for the proposed estimator. Third, we prove that the resulting estimator is regular, asymptotically linear, and semiparametrically efficient under standard conditions. Importantly, efficiency is achieved simultaneously across all target parameters whose canonical gradients lie in the $L^2(P_0)$-closure of the RKHS, including multivariate targets. Finally, we develop a computationally tractable implementation based on finite-dimensional kernel representations and principled stopping criteria derived from the geometry of the universal least favorable flow. 
%Simulation studies demonstrate that ULFS--KDPE exhibits improved finite-sample stability and accuracy relative to iterative local targeting methods, particularly in challenging regimes such as limited overlap.

\subsection{Related Work}

Our work relates to two broad classes of semiparametric debiasing methods: influence-function–based approaches and influence-function–free computational approaches.

Classical influence-function–based methods include estimating equations, one-step estimators, double machine learning (DML), and targeted maximum likelihood estimation (TMLE). Estimating equation approaches identify the target parameter by solving score equations derived from the EIF \cite{robins1986,newey1990,bickel1998,vanderLaanRobins2003,tsiatis2006}. One-step estimators correct an initial estimator by adding the empirical mean of the EIF, yielding asymptotic efficiency under suitable regularity conditions \cite{bickel1998}. Double Machine Learning constructs a Neyman-orthogonal estimating equation based on the EIF and combines it with flexible machine-learned nuisance estimators via sample splitting \cite{chernozhukov2017double}. TMLE produces a plug-in estimator by fluctuating an initial estimate of the data-generating distribution along a locally least favorable submodel until the EIF estimating equation is approximately solved \cite{vanderLaanRubin2006TMLE,book2011,book2018}. While these methods are asymptotically efficient under suitable conditions, least favorability is typically enforced only locally and practical implementations often require iterative targeting steps. Universal least favorable submodels address this limitation by enforcing least favorability globally along an entire distributional path, thereby achieving the desired bias reduction with minimal likelihood fluctuation \cite{one_step_tmle}. A common feature of these approaches, however, is their reliance on explicit derivation and evaluation of the EIF, which can be analytically demanding and parameter-specific in complex semiparametric models.

Influence-function–free approaches aim to bypass analytical derivation of the EIF. Existing methods include approaches that approximate the EIF through finite-difference perturbations \cite{carone2018toward,jordan2022empirical} or Monte Carlo approximations \cite{agrawal2024automated}, as well as automatic debiased machine learning (AutoDML) procedures \cite{chernozhukov2022automatic,vanderlaan2025automaticdebiasedmachinelearning}. Finite-difference and Monte Carlo approaches estimate the EIF numerically, which may introduce additional approximation error. AutoDML instead relies on orthogonal reparameterizations that enable automated debiasing in certain structured settings \cite{chernozhukov2022automatic}. The Kernel Debiased Plug-in Estimator (KDPE) of Cho et al.\ provides a flexible alternative by embedding the debiasing problem in a reproducing kernel Hilbert space \cite{cho2024kernel}. KDPE constructs data-adaptive fluctuations that approximately solve empirical score equations in a universal RKHS without requiring explicit expressions for the EIF. However, KDPE relies on locally defined updates and is typically implemented through iterative targeting steps.

Our approach is also related to the undersmoothed highly adaptive lasso maximum likelihood estimator (HAL-MLE) \cite{laan_undersmooth2022}. The HAL-MLE solves score equations over a rich class of cadlag functions with bounded $\ell_1$ norm. Like HAL, KDPE-based methods solve a large collection of score equations rather than targeting a single parameter-specific score, thereby enabling $\sqrt{n}$-rate bias correction. However, HAL can become computationally challenging in high dimensions because the number of basis functions grows rapidly with the covariate dimension and sample size. In contrast, KDPE-based methods rely on $n$ kernel basis functions centered at the observed data points, yielding substantially improved computational scalability. Recent work has reformulated the original HAL estimator as a Highly Adaptive Ridge (HAR) and PCA-HAR estimator, replacing the $\ell_1$ penalty with an $\ell_2$ ridge penalty and exploiting kernelized representations of the design matrix \cite{schuler2024highlyadaptiveridge,wang2026highlyadaptiveprincipalcomponent}. These approaches also admit RKHS formulations and achieve similar dimension-free convergence rates (up to logarithmic factors), but under stronger smoothness assumptions such as square-integrable sectional derivatives \cite{schuler2024highlyadaptiveridge}.

\section{Problem Setup and Notation}\label{sec::Notation}

Let $d,n\in\mathbb{N}$ be fixed positive integers, with $d$ treated as fixed and not depending on $n$. For simplicity of exposition, we define the covariate space $\X=[0,1]^d$, the treatment space $\A=\{0,1\}$, and the outcome space $\Y=\{0,1\}$. The sample space is $\O = \X \times \A \times\Y$, equipped with the product $\sigma$-algebra $\mathcal{B}(\O)=\mathcal{B}([0,1]^d) \otimes 2^{\A} \otimes 2^{\Y}$. Let $O=(X,A,Y)$ be a random element of $\O$ defined on  some underlying probability space $(\Omega, \mathcal{F}, \mathbb{P})$. For a sample of size $n$, we observe independent and identically distributed random variables $O_1,\ldots, O_n$, where $O_i=(X_i,A_i,Y_i)$. The theory extends directly to bounded measurable sample spaces $\O\subset\mathbb R^k$.

We consider a nonparametric statistical model $\M$ consisting of probability measures on $(\O,\mathcal{B}(\O))$ that are dominated by the $\sigma$-finite product measure $\mu = \lambda_d \otimes \mu_A \otimes \mu_Y$, where $\lambda_d$ denotes Lebesgue measure on $\X=[0,1]^d$ and $\mu_A,\mu_Y$ denote counting measures on $\A$ and $\Y$, respectively. Let $P^*$ denote the true data-generating distribution of $O$, and assume that $P^* \in \M$. The superscript “$*$” will be used throughout to denote true, unknown features of the data-generating mechanism. For $P \in \M$, write $p = dP/d\mu$ for its Radon–Nikodym derivative with respect to $\mu$. Let $\bM := \{p \ge 0 : \int pd\mu = 1\}$ denote the set of densities corresponding to distributions in $\M$.
%We define $\bM := \{p : p=dP/d\mu \text{ for some } P \in \M\}$, where densities are identified up to $\mu$-almost everywhere equivalence.

Realizations of $O$ are denoted by $o \in \O$. Let $P_X^*$ denote the marginal distribution of $X$ under $P^*$, and write $p_{X}^* = dP_X^*/d\lambda_d$ for its density with respect to Lebesgue measure on $\X=[0,1]^d$. Then the joint density $p^* = dP^*/d\mu$ admits the following factorization
\begin{equation}\label{eq::factorization}
p^*(x,a,y)=p_{X}^*(x) e^*(a\mid x) q^*(y\mid a,x),
\end{equation}
for $\mu$-almost every $(x,a,y) \in \O$, where $e^*(\cdot\mid x)$ and $q^*(\cdot\mid a,x)$ denote the true conditional probability mass functions of $A$ given $X=x$ and $Y$ given $(A=a,X=x)$, respectively, viewed as densities with respect to counting measure.
Further, let $t\mapsto p_t$ be a path of probability densities with respect to $\mu$, indexed by $t$ in an open interval $I\subset\mathbb R$, such that $p_t \in \bM$ for all $t\in I$.
%The score of the path at $t=0$ is $h = \left.\frac{d}{dt}\log p_t\right|_{t=0}$, which satisfies $\E_P h=0$ and $\E_P h^2<\infty$.

\vspace{2mm}
\noindent
\textbf{Notation:} For a measurable function $f : \O \to \R$ with $P|f| < \infty$, we write $Pf = P[f] := \int f \, dP = \E_P[f]$. The empirical measure is denoted by $P_n$, defined by $P_n f := \frac{1}{n}\sum_{i=1}^n f(O_i)$. Let $L^2(P)$ denote the Hilbert space of measurable functions $f : \O \to \R$ satisfying $P[f^2]<\infty$, identified up to $P$-almost everywhere equality, and equipped with inner product $\langle f,g\rangle_{L^2(P)} = P[fg]$. We define $L_0^2(P) := \{ f \in L^2(P) : P[f]=0 \}$ as the $L^2(P)$ subspace consisting of mean-zero functions.

\subsection{Target Parameter}\label{sub::Target}

Let $\Psi : \M \to \R^m$ denote an $m$-dimensional parameter mapping defined on the statistical model $\M$. For simplicity we focus on the scalar case $m=1$, and write $\psi^*=\Psi(P^*)$ for the target parameter of interest. Our goal is to efficiently estimate $\psi^*$ on the basis of $n$ i.i.d. observations. 
%To this end, we consider estimators that can be written as functionals of the empirical measure, $\psi_n=\hat\Psi(P_n)$. In many settings of interest, $\hat\Psi$ admits the representation $\hat\Psi(P_n)=\Psi(\hat P_n)$, where $\hat P_n$ is a suitably constructed data-dependent estimate of $P^*$ in $\M$. 
We focus on plug-in estimators of the form $\psi_n=\Psi(\hat P_n)$ where $\hat P_n$ is a data-dependent estimate of $P^*$ in $\M$. This formulation encompasses a broad class of estimators within a unified framework while retaining a plug-in structure.

\section{Preliminaries}\label{sec::Prelim}

We restrict attention to parameters that are pathwise differentiable with respect to $\M$, ensuring first-order sensitivity to smooth perturbations of the underlying distribution. Pathwise differentiability is necessary for the existence of regular estimators and forms the basis of semiparametric efficiency theory through the canonical gradient.

\begin{definition}[Pathwise Differentiability]\label{def::PD}
A parameter $\Psi:\M\to\R$ is said to be \emph{pathwise differentiable at $P\in\M$} if there exists $\phi_P^*\in L_0^2(P)$ such that for every regular parametric submodel $\{P_\epsilon:\epsilon\in(-l,l)\}\subset\M$ satisfying $P_{\epsilon=0}=P$ and differentiability in quadratic mean at $\epsilon=0$ with score
$$
h=\left.\frac{d}{d\epsilon}\log p_\epsilon\right|_{\epsilon=0}\in L_0^2(P),
$$
the map $\epsilon\mapsto\Psi(P_\epsilon)$ is differentiable at $\epsilon=0$ and
$$
\left.\frac{d}{d\epsilon}\Psi(P_\epsilon)\right|_{\epsilon=0}
=
P[\phi_P^* h].
$$
The function $\phi_P^*$ is called the canonical gradient of $\Psi$ at $P$, or the efficient influence function (EIF) relative to the model $\M$.
\end{definition}

We focus on regular and asymptotically linear (RAL) estimators, a broad and tractable class that includes efficient estimators for pathwise differentiable parameters under standard regularity conditions \cite{van2000asymptotic}. Asymptotic linearity yields first-order behavior and asymptotic normality via an influence function, while regularity ensures stability under local perturbations. 
%When a regular estimator attains the efficiency bound, its influence function equals the EIF $\phi_{P^*}^*$. 
We formalize this notion next.

\begin{definition}[Asymptotic Linearity]\label{def::AL}
An estimator $\psi_n=\Psi(\hat P_n)$ of $\psi^*=\Psi(P^*)$ is {asymptotically linear at $P^*$} if there exists a function
$\phi_{P^*}\in L_0^2(P^*)$ such that
$$
\sqrt{n}\Bigl[(\psi_n-\psi^*)-P_n\phi_{P^*}\Bigr]\xrightarrow{P}0.
$$
The function $\phi_{P^*}$ is called the influence function of the estimator.
\end{definition}

\begin{definition}[Regular Estimator]\label{def::reg}
An estimator $\psi_n$ of $\psi^*=\Psi(P^*)$ is said to be \emph{regular at $P^*\in\M$} if for every regular parametric submodel $\{P_\epsilon:\epsilon\in(-l,l)\}\subset\M$ satisfying $P_{\epsilon=0}=P^*$ and differentiability in quadratic mean at $\epsilon=0$ with score $h\in L_0^2(P^*)$, the sequence
$$
\sqrt{n}\bigl(\psi_n-\Psi(P_{\epsilon_n})\bigr)
$$
converges in distribution to a non-degenerate limit that does not depend on the direction $h$, for local alternatives $\epsilon_n=t/\sqrt{n}$ with $t\in\mathbb R$.
\end{definition}

Pathwise differentiability, asymptotic linearity, and regularity together characterize the semiparametric efficiency framework. If $\Psi$ is pathwise differentiable at $P^*$ with canonical gradient $\phi_{P^*}^*$, and an estimator $\psi_n$ is regular and asymptotically linear with influence function $\phi_{P^*}$, then the asymptotic variance of $\psi_n$ is $P^*[\phi_{P^*}^2]$. Semiparametric efficiency theory implies that $P^*[\phi_{P^*}^2] \ge P^*[(\phi_{P^*}^*)^2]$, with equality if and only if $\phi_{P^*} = \phi_{P^*}^*$ almost surely. In this case the estimator is said to be efficient in the model $\M$ \cite{bickel1998,van2000asymptotic}.

\subsection{Scores}

Let $h=\left.\frac{d}{d\epsilon}\log p_{\epsilon}\right|_{\epsilon=0}$ denote the score of a regular one-dimensional parametric submodel $\{P_\epsilon:\epsilon\in(-l,l)\}\subset\M$ through $P$. Differentiability in quadratic mean implies $h\in L_0^2(P)$ \cite{van2000asymptotic}. The closure in $L_0^2(P)$ of the linear span of all such scores defines the \emph{tangent space} of $\M$ at $P$, which coincides with $L_0^2(P)$ in the fully nonparametric model. Given $h\in L_0^2(P)$, common submodels realizing this score include the linear path $p_{\epsilon,h}=(1+\epsilon h)p$, defined for sufficiently small $|\epsilon|$ such that $1+\epsilon h \ge 0$, and the exponential tilting path $p_{\epsilon,h}=\exp(\epsilon h)p/P[\exp(\epsilon h)]$, 
both satisfying $\left.\frac{d}{d\epsilon}\log p_{\epsilon,h}\right|_{\epsilon=0}=h$. While sufficient for local efficiency theory, these constructions are not suitable for defining universal least favorable paths.

\subsection{Influence function}
Pathwise differentiability of $\Psi$ at $P$ implies that its directional derivative along any submodel with score $h\in L_0^2(P)$ is a continuous linear functional of $h$. By the Riesz representation theorem, there exists a unique $\phi_P^*\in L_0^2(P)$ such that $\left.\frac{d}{d\epsilon}\Psi(P_{\epsilon,h})\right|_{\epsilon=0}=P[\phi_P^*h],$
for all $h$ in the tangent space (which equals $L_0^2(P)$ in the nonparametric model). An estimator admits an influence function $\phi_{P^*}$ through asymptotic linearity. If the estimator is regular and efficient, then $\phi_{P^*}=\phi_{P^*}^*$ almost surely.

We analyze the estimation error $\Psi(\hat P_n)-\Psi(P^*)$ using a von Mises–type expansion, which provides the foundation for first-order asymptotic analysis of RAL estimators \cite{bickel1998}, as shown in \Cref{prop::vmHatP}. Rather than expanding $\Psi(\cdot)$ around $P^*$, we work with an expansion around the data-adaptive estimate $\hat P_n$, which isolates empirical-process fluctuations and higher-order remainder terms in a form convenient for analyzing targeted estimators.

\begin{proposition}[von Mises expansion {\citep{bickel1998}}]\label{prop::vmHatP}
Assume that $\Psi$ is pathwise differentiable at both $P^*$ and $\hat P_n$, with canonical gradients $\phi_{P^*}^*\in L_0^2(P^*)$ and $\phi_{\hat P_n}^*\in L_0^2(\hat P_n)$. Suppose that $\Psi$ admits a second-order remainder $R_2(\cdot,\cdot)$ in a neighborhood of $\hat P_n$, so that
$$
\Psi(P^*)-\Psi(\hat P_n)=(P^*-\hat P_n)\phi_{\hat P_n}^*+R_2(P^*,\hat P_n).
$$
Then the estimation error $\Psi(\hat P_n)-\Psi(P^*)$ decomposes as
\begin{align}\label{eq::asyExpansionHatP}
{\mathbb P_n\phi_{P^*}^*}- \underbrace{\mathbb P_n\phi_{\hat P_n}^*}_{Term \ 1}
+ \underbrace{(\mathbb P_n-P^*)\bigl(\phi_{\hat P_n}^*-\phi_{P^*}^*\bigr)}_{Term \ 2}
+ \underbrace{R_2(P^*,\hat P_n).}_{Term \ 3}
\end{align}
%which separates the leading empirical term, the plug-in bias, an empirical-process remainder, and a second-order remainder.
\end{proposition}

The expansion in Proposition~\ref{prop::vmHatP} yields asymptotic linearity provided the empirical-process ($Term \ 2$), the second-order remainder ($Term \ 3$), and the plug-in bias ($Term \ 1$) are controlled. The first two are handled by standard empirical-process and remainder conditions stated below, which are commonly assumed in the causal inference literature \cite{book2011, book2018}. The goal of the proposed procedure is to construct $\hat P_n$ so that $\phi_{\hat P_n}^*$ is small in empirical mean, thereby eliminating the leading plug-in bias term in the expansion above.

\begin{assumption}\label{ass::emprocess}
Assume there exists a deterministic function class $\mathcal F\subset L_0^2(P^*)$ such that $P^*\{\phi_{\hat P_n}^*\in\mathcal F\}\to1$, the class $\mathcal F$ is $P^*$-Donsker, and $\|\phi_{\hat P_n}^*-\phi_{P^*}^*\|_{L^2(P^*)}\to_{P^*}0$. Then the empirical-process term satisfies $(\mathbb P_n-P^*)\bigl(\phi_{\hat P_n}^*-\phi_{P^*}^*\bigr) = o_{P^*}(n^{-1/2})$.
\end{assumption}

\begin{assumption}\label{ass::2orderrem}
Assume $\hat P_n$ converges sufficiently fast and $\Psi$ is sufficiently regular so that the second-order remainder in Proposition~\ref{prop::vmHatP} satisfies $R_2(\hat P_n,P^*)=o_{P^*}(n^{-1/2}).$
\end{assumption}

\subsection{Least favorable submodels}\label{sub::submodels}

Classical semiparametrics introduces distributional paths through \emph{locally least favorable submodels} (LLFS), whose score at the initial distribution coincides with the canonical gradient of the parameter \cite{book2011, book2018}. An LLFS is a one-dimensional parametric submodel $\{P_\varepsilon\}$ through $P$ whose score at $\varepsilon=0$ equals the canonical gradient $\phi_P^*$ of $\Psi$ at $P$. This construction is purely local: least favorability holds only infinitesimally at the single distribution $P$. 

\begin{definition}[Locally least favorable submodel]\label{def::LLFS}
Let $P \in \M$ with density $p = dP/d\mu$, and let
$\Psi : \M \to \R$ be pathwise differentiable at $P$ with canonical gradient $\phi_P^*\in L_0^2(P)$. A one-dimensional parametric submodel $\{P_\varepsilon:\varepsilon\in(-\delta,\delta)\}\subset\M$ for $\delta > 0$ is called a \emph{locally least favorable submodel}
for $\Psi$ at $P$ if:

\begin{enumerate}
\item $P_{\varepsilon=0}=P$;
\item the submodel is differentiable in quadratic mean at $\varepsilon=0$;
\item the score of the submodel at $\varepsilon = 0$ satisfies
\begin{equation}\label{eq::LLFS_score}
\left.\frac{d}{d\varepsilon}\log p_\varepsilon(o)\right|_{\varepsilon=0}
= \phi_P^*(o), \qquad P\text{-a.e. }o\in\O.
\end{equation}
\end{enumerate}
\end{definition}

In contrast, the \emph{universal least favorable submodel} (ULFS) enforces least favorability uniformly along an entire path \cite{one_step_tmle}. A ULFS through $P$ is a parametric submodel $\{\Pt:t\in I\}$ whose score at each point along the path coincides with the canonical gradient evaluated at the current distribution. 
%We allocate a more detailed description of least favorable submodels to the Appendix \Cref{sub::submodels}.

\begin{definition}[Universal least favorable submodel]\label{def::ULFS}
Let $P \in \M$ with density $p = dP/d\mu$, and let
$\Psi : \M \to \R$ be pathwise differentiable on $\M$ with canonical gradient $\phi_P^*$. A parametric submodel $\mathrm{ULFS}(P)=\{P_t:t\in I\}\subset\M$, where $I\subset\R$ is open, is called a \emph{universal least favorable submodel} through $P$ if:

\begin{enumerate}
\item $P_{t=0}=P$;
\item for each $t\in I$, the map $t\mapsto p_t$ is differentiable;
\item for each $t\in I$,
\begin{equation}\label{eq::def_ULFS_ode}
\frac{d}{dt}\log p_t(o)=\phi_{P_t}^*(o),
\qquad \mu\text{-a.e. }o\in\O,
\end{equation}
where $\phi_{P_t}^*$ denotes the canonical gradient of $\Psi$ evaluated at $P_t$.
\end{enumerate}
\end{definition}

Equation~\eqref{eq::def_ULFS_ode} is the differential equation defining the ULFS, with initial condition $p_0=p$. Assuming the path is well defined and normalizable, integration yields the normalized representation
\begin{equation}\label{eq::def_ULFS_explicit}
p_t(o)=
\frac{
p(o)\exp\!\left(\int_0^t \phi_{P_x}^*(o)\,dx\right)
}{
P\!\left[\exp\!\left(\int_0^t \phi_{P_x}^*\,dx\right)\right]
}.
\end{equation}
Conversely, any path of the form \eqref{eq::def_ULFS_explicit} that is differentiable in $t$ satisfies \eqref{eq::def_ULFS_ode}. Thus, the ULFS may be viewed either as the solution to the ODE~\eqref{eq::def_ULFS_ode} or as the exponential flow generated by the canonical gradient evaluated along the path.

%The essential distinction between locally least favorable submodels and the ULFS is therefore geometric. Locally least favorable submodels identify tangent directions that are optimal only at an initial distribution $P$ and only infinitesimally at that point, typically for a specified target parameter. In contrast, the ULFS defines a globally evolving parametric path ${\Pt:t\in I}$ whose score coincides with the relevant canonical gradient at each distribution $\Pt$ along the path. In this sense, least favorability is enforced uniformly along the entire submodel, rather than only locally at the initial distribution. 

\subsection{Reproducing Kernel Hilbert Space (RKHS)}\label{sub::RKHS}

Let $\V$ be a non-empty set. A function $K:\V\times\V\to\mathbb R$ is called positive definite (PD) if, for any finite collection $\bfv=(v_1,\ldots,v_n)\in\V^n$ and any $\alpha\in\mathbb R^n\setminus\{0\}$, $\sum_{i,j=1}^n \alpha_i\alpha_j K(v_i,v_j)>0$\footnote{For positive semidefinite, which is often all that is needed, $\sum_{i=1}^n\sum_{j=1}^n \alpha_i \alpha_j K(v_i,v_j) \ge 0$.}. For $v\in\V$, define the kernel section $k_v(\cdot)=K(v,\cdot)$. Given finite
collections $\bfv\in\V^n$ and $\bfv'\in\V^m$, the associated kernel matrix is $K_{\bfv\bfv'}=[K(v_i,v'_j)]\in\mathbb R^{n\times m}$; in particular, $K_{\bfv}:=K_{\bfv\bfv}$ is symmetric and PD. Equivalently, collecting kernel sections as $k_{\bfv}(\cdot)=(k_{v_1}(\cdot),\ldots,k_{v_n}(\cdot))^{\mathsf T}$ yields
$k_{\bfv}(\bfv')=K_{\bfv\bfv'}$.

Given a PD kernel $K$ as defined above, a natural question is whether there exists a canonical Hilbert space of functions on $\V$ in which the kernel sections $\{k_v:v\in\V\}$ represent point evaluation. The
answer is provided by Moore-Aronszajn \citep{aronszajn50reproducing}, restated in \Cref{prop::MooreAronszajn}, which shows the existence and uniqueness of the reproducing kernel Hilbert space (RKHS) associated with $K$. 

\begin{proposition}[Moore--Aronszajn \cite{aronszajn50reproducing}]\label{prop::MooreAronszajn}
Let $\V$ be a non-empty set and let $K:\V\times\V\to\mathbb R$ be a positive definite kernel. Then there exists a unique Hilbert space $\H_K$ of functions $f:\V\to\mathbb R$ such that:
\begin{enumerate}
\item for every $v\in\V$, the kernel section $k_v(\cdot)=K(v,\cdot)$ belongs
to $\H_K$;
\item the reproducing property holds, namely,
\begin{equation*}
 \langle f, k_v \rangle_{\H_K} = f(v)
\qquad \text{for all } f\in\H_K \text{ and } v\in\V.   
\end{equation*}
\end{enumerate}
Moreover, $\H_K$ is the completion of the linear span of
$\{k_v : v\in\V\}$ under the inner product induced by $K$.
\end{proposition}

We refer to $\H_K$ as the \emph{reproducing kernel Hilbert space} (RKHS) associated with $K$, and denote its inner product and norm by
$\langle\cdot,\cdot\rangle_{\H_K}$ and $\|\cdot\|_{\H_K}$.
By definition of the RKHS inner product, the kernel sections satisfy $\langle k_v,k_{v'}\rangle_{\H_K}=K(v,v')$. Two structural properties of RKHSs will be used repeatedly. First, point evaluation is a continuous linear functional: for each $v\in\V$, $f(v)=\langle f,k_v\rangle_{\H_K}$, $f\in\H_K$. Second, $\H_K$ is the Hilbert-space completion of $\operatorname{span}\{k_v:v\in\V\}$, so every $f\in\H_K$ can be approximated in $\|\cdot\|_{\H_K}$ by finite linear combinations of kernel sections. If the kernel is bounded, $\sup_{v\in\V}K(v,v)<\infty$, then for every $f\in\H_K$,
\[
\|f\|_\infty
=
\sup_{v\in\V}|f(v)|
\le
\|f\|_{\H_K}
\Bigl(\sup_{v\in\V}K(v,v)\Bigr)^{1/2}.
\]
%Consequently, convergence in $\|\cdot\|_{\H_K}$ implies uniform convergence on $\V$.

\subsubsection{Mean-zero RKHS}\label{subsub::MeanZeroRKHS}

Let $K$ be a PD kernel on $\V$ with associated RKHS $\H_K$, and let
$P$ be a probability measure on $\V$. When the kernel mean embedding $m_P := \int k_v\,dP(v)$ exists as an element of $\H_K$, expectation is represented as an inner product on $\H_K$, which leads to a natural $P$-mean-zero subspace.

\begin{proposition}[Mean-zero RKHS]\label{prop::meanZeroRKHS}
Let $K$ be a positive-definite kernel on $\V$ with associated RKHS $\H_K$, and let $P$ be a probability measure on $\V$ such that the kernel mean embedding
$m_P:=\int k_v\,dP(v)$ exists in $\H_K$. Define the $P$-mean-zero subspace
\begin{equation*}
    \H_{K,P} := \{g \in \H_K: P[g]=0\}.
\end{equation*}
Then $\H_{K,P}$ is a closed linear subspace of $\H_K$, and therefore a Hilbert space under the inherited inner product.
\end{proposition}

Let $\Pi_P : \H_K \to \H_{K,P}$ denote the orthogonal projection. For each $v \in \V$, the representer of evaluation at $v$ in $\H_{K,P}$ is
\begin{equation*}
k_{P,v} := \Pi_P k_v
= k_v - \frac{\langle k_v,m_P\rangle_{\H_K}}{\|m_P\|_{\H_K}^2}\,m_P
= k_v - \frac{m_P(v)}{\|m_P\|_{\H_K}^2}\,m_P,
\end{equation*}
and the centered kernel admits the equivalent representation
\begin{equation*}
K_P(v,v')
=\langle \Pi_P k_v,\Pi_P k_{v'}\rangle_{\H_K}
= K(v,v')-\frac{m_P(v)\,m_P(v')}{\|m_P\|_{\H_K}^2}.    \end{equation*}
The positive definiteness of $K_P$ follows as the reproducing kernel of $\H_{K,P}$.

\subsubsection{Universal Kernels}\label{subsub::UniversalKernels}

Certain kernels generate RKHSs that can approximate large classes of functions independently of the underlying probability measure. Let $C_0(\V)$ denote the Banach space of real-valued continuous functions on $\V$ that vanish at infinity, equipped with the uniform norm $\|\cdot\|_\infty$.

\begin{definition}[Universal kernel]\label{def::UniversalKernel}
A continuous PD kernel $K$ on $\V$ is called \emph{universal} if its associated RKHS is dense in $C_0(\V)$ with respect to the uniform norm.
\end{definition}

Universality implies that $\H_K$, RKHS of the universal kernek $K$, is dense in $L^p(P)$ for every probability measure $P$ on $\V$ and every $1\le p<\infty$ \cite{micchelli2006universal}. In particular, $\H_K$ is dense in $L_0^2(P)$. A canonical example is the Gaussian kernel on $\V=\R^d$. For fixed $\sigma>0$,
\[
K(v,v')
=
\exp\!\left(-\frac{\|v-v'\|_2^2}{2\sigma^2}\right),
\qquad v,v'\in\R^d.
\]
This kernel is bounded with $K(v,v)=1$, and its associated RKHS $\H_K$ consists of bounded continuous functions. Since the kernel is bounded, the kernel mean embedding $m_P=\int k_v\,dP(v)$ exists in $\H_K$ for every probability measure $P$ on $\R^d$. The Gaussian kernel is universal, and therefore $\H_K$ is dense in $L_0^2(P)$. Consequently, restricting perturbations to lie in $\H_K$ allows us to approximate IFs in $L_0^2(P)$.

\subsection{Fixed Point Analysis}\label{sub::FixedPoint}
 
Let $(X,d)$ be a complete metric space and let $T : X\to X$ be a mapping on the said metric space. Below we provide a few essential Definitions and Theorems underlying the Banach contraction principle. 

\begin{definition}[Contraction]
Let $X$ be a metric space equipped with a distance $d$. A map $T : X \rightarrow X$ is said to be Lipschitz continuous if there exists a constant $\lambda \in (0,1)$ such that 
\begin{equation*}
d\bigl(Tx,Ty\bigr) \le \lambda\,d(x,y) 
\qquad \text{for all } x,y\in X.    
\end{equation*}
For $\lambda < 1$, $T$ is a contraction. If $\lambda \leq 1$, then $T$ is non-expansive. 
\end{definition}
  
The classical Banach fixed point theorem guarantees existence and uniqueness of a fixed point for contractions on complete metric spaces, as stated in Theorem~\ref{thm::bfixedpoint}. A proof may be found, for example, in Corollary 1.3 of \cite{pata2019fixed}.

\begin{theorem}[Banach fixed point]\label{thm::bfixedpoint}
Let $T$ be a contraction on a complete metric space $(X,d)$. Then $T$ has a unique fixed point $x^* \in X$.
\end{theorem}

The following extension, sometimes referred to as the \emph{eventual contraction principle}, allows the map itself to fail to be a contraction, provided that some finite iterate is. For $T : X \rightarrow X$ and $m \in \mathbb N$, we denote by $T^m$ the $m^{th}$-iterate of $T$, namely $T \circ T \cdots \circ T$ $m$-times, with $T^0$ the identity map. 

\begin{corollary}[Fixed point via iterated contraction]\label{corr::fixedpoint}
Let $(X,d)$ be a complete metric space and let $T : X\to X$ be a mapping. Suppose there exists an integer $m \ge 1$ such that the iterate $T^m$ is a contraction on $X$. Then $T$ admits a unique fixed point $x^*\in X$, i.e., $T(x^*)=x^*$. Moreover, for any initial point $x_0\in X$, the sequence of iterates $\{T^k(x_0)\}_{k \ge 0}$ converges to $x^*$.    
\end{corollary}

We note that Corollary~\ref{corr::fixedpoint} follows immediately from the Banach fixed point theorem applied to $T^m$, noting that any fixed point of $T^m$ is necessarily a fixed point of $T$, and that uniqueness of the fixed point of $T^m$ implies uniqueness for $T$. 

\begin{theorem}[Picard-Lindelöf theorem in Banach spaces \cite{teschl2012ordinary}]\label{thm::picard}
Let $(\mathcal X,\|\cdot\|)$ be a Banach space and let $U\subset \mathcal X$ be open.
Suppose $F:U\to \mathcal X$ is locally Lipschitz. Then for every $x_0\in U$, there exist
$\tau>0$ and a unique solution $x\in C^1((-\tau,\tau);\mathcal X)$ to the initial value problem
\[
\frac{d}{dt}x(t)=F(x(t)),\qquad x(0)=x_0.
\]
Moreover, the solution satisfies the integral equation
\[
x(t)=x_0+\int_0^t F(x(s))\,ds,\qquad t\in(-\tau,\tau).
\]
\end{theorem}

\section{Debiasing with RKHS}\label{sec::debiasRKHS}

Let $t\mapsto p_t$ denote a path of probability densities on $\O$ with respect to a fixed $\sigma$-finite measure $\mu$, and write $\Pt\in\M$ for the associated probability measure. For a universal least favorable submodel (ULFS), the path $\{\Pt:t\in I\}$ is not specified as a finite-dimensional parametric family \cite{one_step_tmle}. Instead, the evolution of the density is defined through a score-type direction $h(\Pt)\in L_0^2(\Pt)$ via the differential equation
\begin{equation}\label{eq::ULFS_ode_general}
\frac{d}{dt}\log p_t(o)=h(\Pt)(o).
\end{equation}
Thus, the ULFS defines a flow on the space of probability measures whose instantaneous direction depends on the current distribution. We refer to $h(\Pt)$ simply as the (score-type) \emph{direction} hereafter.

Motivated by this construction, we introduce an RKHS-restricted surrogate flow that drives empirical estimating equations indexed by an RKHS to zero. Let $\HK$ be a reproducing kernel Hilbert space of real-valued functions on $\O$ with reproducing kernel $K:\O\times\O\to\R$. We consider the Gaussian kernel
\[
K(o,o')=\exp\!\left(-\frac{\|o-o'\|_2^2}{2\sigma^2}\right),
\qquad o,o'\in\O,
\]
for some fixed bandwidth $\sigma>0$. This kernel is continuous, positive definite, and bounded with $K(o,o)=1$. For each $o\in\O$, we denote by $k_o(\cdot)=K(o,\cdot)$ the corresponding kernel section. The following result records basic properties of expectations viewed as functionals on $\HK$. All proofs and derivations for \Cref{sec::debiasRKHS} are allocated to Appendix A (\Cref{append::A}).

\begin{proposition}[Expectation as an RKHS functional]\label{prop::RKHS_expectation}
Let $P$ be any probability measure on $\O$, and let $\HK$ be the RKHS associated with the Gaussian kernel $K$. Then the kernel mean embedding $m_P:=\int k_o\,dP(o)$ exists in $\HK$. Moreover, for every $f\in\HK$,
\[
P[f]=\langle f,m_P\rangle_{\HK}.
\]
Consequently, the map $f\mapsto P[f]$ is a continuous linear functional on $\HK$ satisfying
\[
|P[f]|\le \|f\|_{\HK}.
\]
\end{proposition}

This representation shows that expectations can be expressed as inner products in the RKHS. In particular, the RKHS norm controls the magnitude of expectations. In our setting, functions with small RKHS norm have small expectation under $P$. 
%Consequently, empirical estimating equations of the form $(P_n-\Pt)[f]=0$ can be interpreted as orthogonality conditions in $\HK$.

\begin{remark}
In many applications the kernel can be defined on the covariate space $\X$ rather than the full sample space $\O$, since the infinite-dimensional components of the statistical model are indexed by $X$. All constructions remain valid with $\O$ replaced by $\X$ and $\Pt$ replaced by the marginal distribution of $X$.
\end{remark}

\subsection{Working RKHS-Based Tangent Space}\label{sub::rkhsTangentSpace}

For each $t\in I$, we define the associated mean-zero RKHS
\[
\HKt := \H_{K,\Pt} := \{f\in\HK:\Pt[f]=0\}.
\]
This space consists of functions in $\HK$ with zero expectation under $\Pt$. 
As shown below, $\HKt$ is a closed linear subspace of $\HK$ and therefore itself a Hilbert space. 
Moreover, since point evaluation remains continuous on $\HKt$, it is also an RKHS with reproducing kernel denoted by $\Kt$. For $o\in\O$, we write $\kt_o(\cdot)=\Kt(o,\cdot)$ for the corresponding kernel section.

\begin{proposition}[Mean-zero RKHS as a closed subspace]\label{prop::mean_zero_RKHS}
For each $t\in I$, the space $\HKt$ is a closed linear subspace of $\HK$. Moreover,
\[
\HKt
=
\{f\in\HK:\langle f,m_{\Pt}\rangle_{\HK}=0\}
=
(\operatorname{span}\{m_{\Pt}\})^\perp.
\]
In particular, $\HKt$ is a Hilbert space under the inner product inherited from $\HK$.
\end{proposition}

Let $\Pi_t:\HK\to\HKt$ denote the orthogonal projection onto $\HKt$, which exists because $\HKt$ is a closed subspace of $\HK$. Since $\HKt=(\operatorname{span}\{m_{\Pt}\})^\perp$, this projection removes the component of a function in $\HK$ along the direction $m_{\Pt}$.

\begin{proposition}[Orthogonal projection onto the mean-zero RKHS]\label{prop::projection_kernel}
If $m_{\Pt}\neq 0$, the orthogonal projection $\Pi_t:\HK\to\HKt$ is given by
\[
\Pi_t f
=
f-\frac{\langle f,m_{\Pt}\rangle_{\HK}}{\|m_{\Pt}\|_{\HK}^2}\,m_{\Pt}
=
f-\frac{\Pt[f]}{\|m_{\Pt}\|_{\HK}^2}\,m_{\Pt},
\qquad f\in\HK.
\]
\end{proposition}
The projection also determines the reproducing kernel of $\HKt$. For each $o\in\O$,
\[
\Pi_t k_o
=
k_o
-
\frac{m_{\Pt}(o)}{\|m_{\Pt}\|_{\HK}^2}\,m_{\Pt},
\]
where we used the reproducing property $\langle k_o,m_{\Pt}\rangle_{\HK}=m_{\Pt}(o)$. Consequently the reproducing kernel of $\HKt$ is
\[
\Kt(o,o')
=
\langle \Pi_t k_o,\Pi_t k_{o'}\rangle_{\HK}
=
K(o,o')
-
\frac{m_{\Pt}(o)\,m_{\Pt}(o')}{\|m_{\Pt}\|_{\HK}^2},
\qquad o,o'\in\O.
\]
If $m_{\Pt}=0$, then $\HKt=\HK$ and $\Kt=K$.

\subsection{RKHS-Restricted Debiasing Direction and Induced Flow}\label{sub::rkhsBiasFlow}

Restricting admissible directions to $\HKt$ ensures that any direction $h(\Pt)$ automatically satisfies the mean-zero constraint $\Pt[h(\Pt)]=0$. Consequently, the path $t\mapsto \Pt$ evolves according to a direction field whose value at each $\Pt$ lies in $\HKt$. The central question is:
\emph{How should we choose, for each current state $\Pt$, a direction $h(\Pt)\in\HKt$ that drives the path toward debiasing?}

At a fixed $\Pt$, empirical deviations within $\HKt$ are characterized by the functional
$f\mapsto P_n[f]$ on $\HKt$, since $(P_n-\Pt)[f]=P_n[f]$ for all $f\in\HKt$. Because the Gaussian kernel is bounded, the map $f\mapsto P_n[f]$ is continuous with respect to the $\HK$ norm and therefore defines a bounded linear functional on $\HK$. Its restriction to the closed subspace $\HKt\subset\HK$ is likewise bounded; we denote this restricted functional by
$L_n:\HKt\to\mathbb R$, with $L_n(f)=P_n[f]=\frac1n\sum_{i=1}^n f(O_i)$.

\begin{proposition}[Empirical Riesz representer in the mean-zero RKHS]\label{prop::empirical_riesz}
There exists a unique element $m_n^{(t)}\in\HKt$ such that
\[
P_n[f]=\langle f,m_n^{(t)}\rangle_{\HKt}
\qquad\text{for all }f\in\HKt.
\]
Moreover,
\[
m_n^{(t)}=\frac1n\sum_{i=1}^n \kt_{O_i}.
\]
\end{proposition}

By construction of $\HKt$, the population mean embedding satisfies $\int \kt_o\,d\Pt(o)=0$. Thus $m_n^{(t)}$ represents the Riesz representer of the empirical deviation functional $f\mapsto(P_n-\Pt)[f]$ on $\HKt$. Consequently, for any $f\in\HKt$,
\[
|P_n[f]|
=
|\langle f,m_n^{(t)}\rangle_{\HKt}|
\le
\|f\|_{\HKt}\,\|m_n^{(t)}\|_{\HKt},
\]
with equality attained at $f=m_n^{(t)}/\|m_n^{(t)}\|_{\HKt}$ whenever $m_n^{(t)}\neq0$. Therefore
\begin{equation*}
\sup_{\substack{f\in\HKt\\\|f\|_{\HKt}\le1}}|P_n[f]|
=\sup_{\substack{f\in\HKt\\\|f\|_{\HKt}\le1}}
|\langle f,m_n^{(t)}\rangle_{\HKt}|
=\|m_n^{(t)}\|_{\HKt}.    
\end{equation*}
We note several important consequences of this representation. First, $\|m_n^{(t)}\|_{\HKt}$ quantifies the magnitude of empirical moment deviations uniformly over the unit ball of $\HKt$. In particular, $\|m_n^{(t)}\|_{\HKt}$ is the worst-case empirical deviation over the unit ball of $\HKt$. Second, driving $\|m_n^{(t)}\|_{\HKt}$ toward zero forces $P_n[f]$ to be uniformly small over $\{f\in\HKt:\|f\|_{\HKt}\le1\}$. This is the RKHS analogue of approximately solving a continuum of estimating equations indexed by $f\in\HKt$.

\section{One-Step KDPE (ULFS--KDPE)}\label{sec::onestepKDPE}

Fix $t\in I$ and recall that the empirical mean embedding $m_n^{(t)}\in\HKt$ is defined by $m_n^{(t)}=P_n[\kt_O]=\frac1n\sum_{i=1}^n \kt_{O_i}$. By construction, $m_n^{(t)}$ is the Riesz representer of the bounded linear functional
$f\mapsto P_n[f]$ on $\HKt$. Although $\HKt$ is infinite-dimensional, the empirical mean embedding lies in the finite-dimensional subspace $\operatorname{span}\{\kt_{O_1},\ldots,\kt_{O_n}\}\subset\HKt$. We can therefore represent $m_n^{(t)}$ through its evaluations at the observed sample points. Define the coordinate vector $\valpha^{(t)}\in\mathbb R^n$ by
\begin{equation}\label{eq::alpha_def}
\alpha_j^{(t)}
:=
P_n[\kt_{O_j}]
=
\frac1n\sum_{i=1}^n \Kt(O_i,O_j),
\qquad 1\le j\le n.
\end{equation}
Equivalently, by the reproducing property in $\HKt$,
\[
\alpha_j^{(t)}
=
\langle \kt_{O_j},m_n^{(t)}\rangle_{\HKt}
=
m_n^{(t)}(O_j),
\]
so $\valpha^{(t)}$ records the evaluations of the empirical mean embedding at the observed sample points.

The empirical mean embedding $m_n^{(t)}$ provides a complete representation of empirical moment deviations on $\HKt$: $m_n^{(t)}=0$ if and only if $P_n[f]=0$ for all $f\in\HKt$. Consequently, to debias the empirical distribution along a least favorable path, it is natural to define a (score-type) direction that drives $m_n^{(t)}$ toward zero while remaining in $\HKt$. Rather than updating directly in the direction of $m_n^{(t)}$, we apply an empirical preconditioning that aligns the update with the geometry induced by kernel evaluations at the observed sample points. Specifically, define
\begin{equation}\label{eq::Dt_def}
D_t = D(p_t)
:=
\frac1n\sum_{j=1}^n \alpha_j^{(t)}\kt_{O_j}
=
\frac1n\sum_{j=1}^n
\langle \kt_{O_j},m_n^{(t)}\rangle_{\HKt}\,\kt_{O_j}
\in\HKt.
\end{equation}
Equivalently,
\[
D_t = D(p_t)
=
\left(\frac1n\sum_{j=1}^n \kt_{O_j}\otimes\kt_{O_j}\right)m_n^{(t)}
=
\widehat C_t\,m_n^{(t)},
\]
where $\widehat C_t:=\frac1n\sum_{j=1}^n \kt_{O_j}\otimes\kt_{O_j}$ is the empirical covariance operator on $\HKt$.
Using this direction, the ULFS equation \eqref{eq::ULFS_ode_general} becomes
\begin{equation}\label{eq::ode}
\frac{d}{dt}\log p_t(o)=D(p_t)(o).
\end{equation}
Since $D(p_t)\in\HKt$, we have $\Pt[D(p_t)]=0$, which ensures that the normalization of $p_t$ is preserved along the flow.

Intuitively, the empirical mean embedding $m_n^{(t)}$ measures the current empirical bias in $\HKt$. 
It is the steepest descent direction in the RKHS norm, and corresponds to minimizing $\|m_n^{(t)}\|^2_{\HKt}$. The empirical covariance operator $\widehat C_t$ describes how this bias is perceived through kernel evaluations at the observed data points. The update direction $D(p_t)=\widehat C_t m_n^{(t)}$
therefore is a \emph{kernel natural gradient flow}: the steepest descent in the empirical moment geometry. The update therefore emphasizes directions that most strongly reduce empirical moment violations at the observed data. 
%It moves the model along a covariance-preconditioned gradient flow that reduces empirical moment deviations, and each observation pushes the model in proportion to how biased the model currently is at that observation. 
The flow vanishes exactly when the empirical bias disappears, a property we establish later in \Cref{lemma::monotonicity}.

\subsection{ULFS--KDPE Algorithm}

In the following, we describe a practical implementation of the ULFS--KDPE. The method approximates the continuous-time flow in ODE \eqref{eq::ode} by a finite-dimensional discretization induced by kernel evaluations at the observed data. The full algorithm is available in Appendix B (\Cref{append::B}).

\subsubsection{Discretized ULFS update}

Fix a step size $\Delta>0$ and an initial density estimate $\hat p_0$. Given the current iterate $\hat p_t$ and the RKHS-restricted direction $D(\hat p_t)\in\HKt$, we discretize the flow using an explicit Euler step on the log-density
\begin{equation}\label{eq::euler_log}
\log \hat p_{t+\Delta}(o)
=
\log \hat p_t(o)+\Delta\,D(\hat p_t)(o).
\end{equation}
Equivalently,
\begin{equation}\label{eq::euler_mult}
\hat p_{t+\Delta}(o)
=
\hat p_t(o)\exp\!\bigl(\Delta\,D(\hat p_t)(o)\bigr).
\end{equation}
This multiplicative form preserves positivity of the density estimate.

In the continuous-time flow, normalization is preserved because $D(p_t)\in\HKt$ implies $P_t[D(p_t)]=0$. The explicit Euler discretization does not preserve the integral exactly, however, so in practice the update may be followed by a renormalization step:
\begin{equation}\label{eq::euler_renorm}
\hat p_{t+\Delta}(o)
\leftarrow
\frac{\hat p_{t+\Delta}(o)}
{\int \hat p_{t+\Delta}(u)\,d\mu(u)}.
\end{equation}
Thus, each iteration consists of a positivity-preserving exponential tilt followed by normalization.

\subsubsection{What is computed at each iteration?}

At iteration time $t$, the update direction is the RKHS element defined in \Cref{eq::Dt_def}. Thus the discretized update in \Cref{eq::euler_mult} only requires the coefficients
$\valpha^{(t)}$ and the ability to evaluate the kernel sections $\kt_{O_j}(\cdot)$. Operationally, the computations are performed using the centered Gram matrix
associated with the mean-zero kernel $\Kt$,
\begin{equation}\label{eq::G_def}
G^{(t)} = \bigl[\Kt(O_i,O_j)\bigr]_{i,j=1}^n.
\end{equation}
Using \Cref{eq::alpha_def}, the coefficient vector is obtained as
\begin{equation}\label{eq::alpha_matrix}
\valpha^{(t)}
=
\frac1n G^{(t)}\one,
\end{equation}
where $\one$ denotes the $n$-vector of ones. The direction evaluated at the observed sample points is then
\begin{equation}\label{eq::D_on_sample}
\Bigl(D(\hat p_t)(O_i)\Bigr)_{i=1}^n
=
\frac1n G^{(t)}\valpha^{(t)},
\end{equation}
since $D(\hat p_t)(O_i)
=
\frac1n\sum_{j=1}^n \alpha_j^{(t)}\Kt(O_i,O_j)$. Each iteration thus reduces to simple matrix–vector operations involving the Gram matrix $G^{(t)}$.

Starting from an initial density estimate $\hat p_0$ (or distribution estimate $\hat P_0$), the algorithm iteratively:
\begin{enumerate}[label=\roman*]
\item Updates the kernel $\Kt$ and the induced mean-zero RKHS $\HKt$ with respect to the current distribution $\hat P_t$;
\item Computes the empirical mean embedding $m_n^{(t)} = P_n[\kt_O]$, represented through its evaluations at the sample points by the coefficient vector $\valpha^{(t)}$;
\item Forms the RKHS direction $D(\hat p_t)$;
\item Updates the density using the exponential tilt step \eqref{eq::euler_mult}, followed by renormalization \eqref{eq::euler_renorm}.
\end{enumerate}

The procedure continues until the empirical mean embedding is sufficiently small, or an equivalent criterion is met. The final iterate $\hat p_T$ defines the debiased plug-in distribution and $\Psi(\hat P_T)$ is reported as the ULFS-KDPE estimate. 

\subsubsection{Stopping criteria}\label{subsubsec::stopping}

We describe practical stopping rules for the discretized ULFS--KDPE flow. These criteria are intended as complementary diagnostics; no single rule is required for theoretical validity. Let $\{\hat P_t:t=0,\Delta,2\Delta,\ldots\}$ denote the sequence of iterates generated using Euler step size $\Delta$, with corresponding densities $\hat p_t$, empirical mean embeddings $m_n^{(t)}$, and coefficient vectors $\valpha^{(t)}\in\R^n$.

\vspace{0.5em}
\paragraph{\textbf{(SC1) Density plateau:} \textit{Has the fitted density essentially stopped moving?}}
This criterion monitors stabilization of the fitted density itself. Define the squared incremental change in log-density
\[
\Delta_t^{(p)}
:=
P_n\!\left[
\bigl(\log \hat p_t(O)-\log \hat p_{t-\Delta}(O)\bigr)^2
\right].
\]
The algorithm stops if
\[
\Delta_t^{(p)}\le \delta_p
\qquad\text{or}\qquad
\bigl|\Delta_t^{(p)}-\Delta_{t-\Delta}^{(p)}\bigr|\le 0.1\,\delta_p.
\]
Thus SC1 stops the iterations when successive Euler updates produce negligible changes in the fitted distribution.

\vspace{0.5em}

\paragraph{\textbf{(SC2) Score plateau:} \textit{Has the flow stopped increasing the objective in a meaningful way?}}
Because the flow direction is $D(\hat p_t)$, stationarity corresponds to vanishing directional score. Define $s_t:=P_n[D(\hat p_t)]
=\frac{1}{n}\sum_{j=1}^n\bigl(\alpha_j^{(t)}\bigr)^2$.
%where the equality follows from \Cref{lemma::monotonicity}. 
The algorithm stops if
\[
|s_t|\le \delta_s
\qquad\text{or}\qquad
|s_t-s_{t-\Delta}|\le 0.1\,\delta_s.
\]
This criterion is directly tied to the Lyapunov derivative, $\frac{d}{dt}P_n[\log p_t]$, and therefore has a clear theoretical interpretation: once $s_t$ is small, the flow is no longer making substantial progress.

\vspace{0.5em}

\paragraph{\textbf{(SC3) Vanishing RKHS update direction:} \textit{Is the actual RKHS update direction nearly zero?}}
Since $D(\hat p_t)$ is the actual update direction, a small RKHS norm indicates that the flow is close to equilibrium. The algorithm stops if
\[
\frac{1}{n}\|D(\hat p_t)\|_{\HKt}^2
=
\frac{1}{n}\|\widehat C_t m_n^{(t)}\|_{\HKt}^2
\le \delta_\alpha.
\]
This criterion measures the size of the preconditioned bias directly. In practice it is closely connected to the magnitude of the coefficient vector $\valpha^{(t)}$ and is inexpensive to evaluate.

\vspace{0.5em}

\paragraph{\textbf{(SC4) Variance-dominated updates:} \textit{Are we mostly accumulating variability rather than making meaningful progress?}}
To guard against numerical overfitting or instability, we compare the variability of the update with its average improvement. Define
\[
\Delta_t^{(v)}
:=
P_n\!\left[
\bigl(\log \hat p_t(O)-\log \hat p_{t-\Delta}(O)\bigr)^2
\right]
\]
and
\[
\Delta_t^{(\ell)}
:=
\Bigl|
P_n\!\left[
\log \hat p_t(O)-\log \hat p_{t-\Delta}(O)
\right]
\Bigr|.
\]
The algorithm stops if either
\[
\Delta_t^{(v)}\le \delta_v
\]
or
\[
P_n\!\left[
\bigl(\log \hat p_t(O)-\log \hat p_0(O)\bigr)^2
\right]\gtrsim n^{-1}
\qquad\text{and}\qquad
\Delta_t^{(\ell)}\le \delta_\ell.
\]
The first condition detects negligible incremental movement, while the second stops the flow when cumulative variability has become nontrivial but further average improvement is negligible.

\vspace{0.5em}

\paragraph{\textbf{(SC5) EIF approximately solved:} \textit{If the EIF is available, has the usual TMLE score equation essentially been solved?}}
When the EIF $\phi_{\hat P_t}^*$ for the target parameter is available, or can be stably approximated, we may additionally monitor its empirical mean. The algorithm stops if
\[
\bigl|P_n[\phi_{\hat P_t}^*]\bigr|\le c\,n^{-1/2}
\qquad\text{and}\qquad
\bigl|P_n[\phi_{\hat P_t}^*]\bigr|
\ge
\bigl|P_n[\phi_{\hat P_{t-\Delta}}^*]\bigr|.
\]
This criterion mirrors the classical TMLE stopping rule: once the EIF estimating equation is solved to first order and no further improvement is observed, the targeting step is terminated. It is optional here, since ULFS--KDPE is influence-function free by construction.

\section{Theoretical Results}

Instead of working directly with the log-density formulation in
\Cref{eq::ode}, we analyze the equivalent ODE
\begin{equation}\label{eq::newode}
\frac{d}{dt}p_t(o)=p_t(o)\,D(p_t)(o),
\end{equation}
which follows from the identity
$\frac{d}{dt}\log p_t = D(p_t)$ whenever $p_t(o)>0$ and the map
$t\mapsto p_t(o)$ is differentiable. Throughout this section we construct solutions $p_t$ with sufficient regularity so that both formulations are well defined and equivalent.

To place \Cref{eq::newode} in a functional-analytic framework, we write the dynamics in state-indexed form. Let $p$ be a nonnegative function satisfying $\int_\O p\,d\mu = 1$, and define the associated probability
measure $P_p(A) := \int_A p(o)\,d\mu(o)$, $A\subset\O$. When $p=p_t$ we recover $P_p=P_t$. 
%Fix a bounded positive definite kernel $K$ on $\O$ (the Gaussian kernel) with RKHS $\HK$. 
%For each $p$, define the mean-zero RKHS $\H_{K,p} := \{f\in\HK : P_p[f]=0\}$,  and let $\Pi_p:\HK\to\H_{K,p}$ denote the orthogonal projection.
For each $p$, let $\Pi_p:\HK\to\H_{K,p}$ denote the orthogonal projection with mean-zero RKHS $H_{K,p}$. We define the projected kernel section $k_o^{(p)} = \Pi_p k_o$, with $K^{(p)}(o,o') = \langle k_o^{(p)},k_{o'}^{(p)}\rangle_{\HK}$ for $o\in\O$.  Given $O_1,\dots,O_n$, define the coefficient vector $\valpha^{(p)}\in\R^n$ by $\alpha_j^{(p)} =\frac1n\sum_{i=1}^n K^{(p)}(O_i,O_j)$ and the associated flow direction $D(p) := \frac1n\sum_{j=1}^n \alpha_j^{(p)}\,k_{O_j}^{(p)}$. With this notation, the ULFS evolution in \Cref{eq::newode} can be written compactly as the autonomous ODE
\begin{equation}\label{eq::autoode}
\frac{d}{dt}l = F(p),
\qquad
F(p):=p\,D(p),
\end{equation}
which we study as a dynamical system on a suitable function space. The structure of the vector field $F(p)=p\,D(p)$ resembles a replicator-type dynamics; the density evolves multiplicatively in proportion to the direction $D(p)$. Because $D(p)\in\H_{K,p}$ and
therefore satisfies $P_p[D(p)]=0$, the total mass of $p$ is preserved along the flow. Intuitively, the evolution redistributes probability mass across $\O$ according to the RKHS direction $D(p)$ without
altering normalization.

\subsubsection{Hölder structure and notation}\label{sub::holder}

The sample space is $\O = [0,1]^d\times\{0,1\}\times\{0,1\}$, equipped with the product topology. Thus, $\O$ decomposes as the disjoint union of four compact slices
\[
\O_{a,y} = [0,1]^d\times\{a\}\times\{y\}, 
\qquad (a,y)\in\{0,1\}^2,
\]
each homeomorphic to $[0,1]^d$. Consequently, functions on $\O$ may be viewed as collections of four functions, one defined on each slice. Let $O_k=(X_k,A_k,Y_k)$ denote an observation. For $o=(x,A_k,Y_k)\in\O_{A_k,Y_k}$ we define the Euclidean distance along the continuous coordinates by
$\|o-O_k\|_2 := \|x-X_k\|_2$.

Fix $\alpha\in(0,1]$. We define the Banach space $C^{1,\alpha}(\O)
\cong
\bigoplus_{(a,y)\in\{0,1\}^2} C^{1,\alpha}([0,1]^d)$, 
consisting of functions whose restriction to each slice $\O_{a,y}$ lies in $C^{1,\alpha}([0,1]^d)$. The norm on $C^{1,\alpha}(\O)$ is defined by
\[
\|u\|_{C^{1,\alpha}(\O)}
=
\max_{(a,y)\in\{0,1\}^2}
\|u(\cdot,a,y)\|_{C^{1,\alpha}},
\]
where for functions $v$ on $[0,1]^d$,
\[
\|v\|_{C^{1,\alpha}}
=
\|v\|_{L^\infty}
+
\sum_{i=1}^d \|\nabla_i v\|_{L^\infty}
+
\sum_{i=1}^d [\nabla_i v]_\alpha,
\]
and the Hölder seminorm is
\[
[\nabla_i v]_\alpha
=
\sup_{x\neq y}
\frac{|\nabla_i v(x)-\nabla_i v(y)|}{\|x-y\|_2^\alpha}.
\]

Throughout the section we use $\|\cdot\|_{C^{1,\alpha}}$ for the Hölder norm, $\|\cdot\|_\infty$ for the supremum norm, $\|\cdot\|_\alpha$ for the Hölder seminorm, and $\|\cdot\|_2$ for the Euclidean norm, omitting domains when clear from context.

\subsection{Existence and Uniqueness of the ODE Solution}

We show that for any target accuracy $\delta_n>0$ for the empirical RKHS score equation $P_n D(p_t)\le \delta_n$, there exists a finite interval $I=[0,T]$ on which the ODE in \eqref{eq::newode} admits a unique solution. Moreover, the solution reaches the accuracy level $\delta_n$ at some time $t\in[0,T]$. This guarantees that the algorithmic flow is well defined and remains within the class of probability densities. 

The main technical challenge is that the vector field $p\mapsto p\,D(p)$ is data-dependent and acts on the infinite-dimensional Banach space $C^{1,\alpha}(\O)$. Our strategy is to view \Cref{eq::newode} as an autonomous ODE (as in \Cref{eq::autoode}) on a bounded subset of $C^{1,\alpha}(\O)$ and to verify that the induced vector field is locally Lipschitz on that set. Intuitively, boundedness and smoothness of the kernel ensure that RKHS-based updates cannot produce arbitrarily rough or unstable densities over finite time horizons. Let $p_0$ denote the initial density and set $M_0:=\|p_0\|_{C^{1,\alpha}(\O)}$. We consider the closed set
\[
\mathcal B_M :=
\Bigl\{
p\in C^{1,\alpha}(\O):
p\ge 0,\,
\int_\O p\,d\mu=1,\,
\|p\|_{C^{1,\alpha}(\O)}\le M
\Bigr\},
\]
where the radius $M$ will be chosen large enough to contain the entire solution trajectory up to time $T$. The precise dependence of $M$ and $T$ on $\delta_n$, $p_0$, and the kernel constants is given in the proof (Appendix D, \Cref{append::unique}).

We impose regularity conditions on the kernel in \Cref{ass::kernelbound}. These conditions ensure that the RKHS representers appearing in the update $D(p)$ remain uniformly bounded in $C^{1,\alpha}(\O)$ along the trajectory. These conditions are mild and are satisfied by Gaussian kernels on
compact domains, with constants depending only on the bandwidth
$\sigma$, the Hölder exponent $\alpha$, and $\mathrm{diam}(\O)$.

\begin{assumption}
\label{ass::kernelbound}
There exists finite constants $C, M_K, M_{1,\alpha}$, such that
\begin{align}
C &:= \sup_o K(o,o) <\infty, \label{eq:C_def}\\
M_K &:= \sup_{y\in\O}\|K(\cdot,y)\|_{C^{1,\alpha}} <\infty, \label{eq:MK_def}\\
M_{1,\alpha} &:= \sup_{p\in \mathcal{B}_M}\sup_{1\le j\le n}\|k^{(p)}_{O_j}\|_{C^{1,\alpha}} <\infty. \label{eq:M1a_def}
\end{align}  
\end{assumption}

%Consequently, the vector field $F(p):=pD(p)$ maps $\mathcal B_M$ into $C^{1,\alpha}(\O)$ and is locally Lipschitz with respect to the $C^{1,\alpha}$ norm. Standard ODE theory in Banach spaces therefore guarantees existence and uniqueness of solutions starting from $p_0$. Finally, the boundedness of the vector field implies that the $C^{1,\alpha}$ norm of $p_t$ can grow at most exponentially in time. Choosing $M$ sufficiently large ensures that the entire trajectory remains inside $\mathcal B_M$ for $t\in[0,T]$. Under these conditions, the solution exists uniquely on $[0,T]$ and reaches the accuracy level $P_nD(p_t)\le\delta_n$ at some time $t\in[0,T]$.

The following theorem establishes existence, uniqueness, and stability of the ULFS--KDPE flow. In particular, \Cref{t::picard} shows that the density valued ODE in \Cref{eq::newode} has a unique solution starting at $p_0$ and that, over a finite time interval $I = [0,T]$, the solution remains a valid probability density with controlled $C^{1,\alpha}$ norm. The proof is allocated to Appendix D, \Cref{append::unique}.

\begin{theorem}[Existence and uniqueness of the ODE solution]\label{t::picard}
Let $p_0\in C^{1,\alpha}(\O)$ be the initial density. Assume \Cref{ass::kernelbound} holds and let $p_0\in\mathcal B_M\subset C^{1,\alpha}(\O)$. Let $F:\mathcal B_M\to C^{1,\alpha}(\O)$ be defined by $F(p):=p\,D(p)$. Then the density valued ODE in \eqref{eq::newode}, with $p_{t=0}=p_0$,
\begin{equation*}
\frac{d}{dt}p_t = F(p_t),
\end{equation*}
admits a unique solution $t\mapsto p_t\in C^1([0,T];C^{1,\alpha}(\O))$. For all $t\in[0,T]$, $p_t\in\mathcal B_M$.
\end{theorem}

\subsection{Finite-Time Convergence of the Empirical Score}

The following result establishes a key structural property of the RKHS-restricted flow.  The RKHS update direction $D(p_t)$ induces a Lyapunov structure for the flow. In particular, along any sufficiently regular solution, the empirical log-likelihood $P_n[\log p_t]$ is monotone nondecreasing, and stationarity occurs if and only if the empirical mean embedding vanishes. This result both justifies the choice of $D(p_t)$ and provides a principled stopping criterion, 
linking equilibrium points of the ODE to the solution of empirical estimating equations on $\HKt$. The proof is allocated to Appendix D, \Cref{append::convergence}.

\begin{lemma}[Monotonicity and stationarity of the empirical log-likelihood]
\label{lemma::monotonicity}
Let $t \mapsto p_t$ be a sufficiently regular solution to the differential equation 
\begin{equation*}
\frac{d}{dt}\log p_t(o)=D(p_t)(o).
\end{equation*}
with $D(p_t) := \frac1n \sum_{j=1}^n \alpha_j^{(t)}\kt_{O_j}$. Then the following hold:
\begin{enumerate}
\item The map $t \mapsto P_n[\log p_t]$ is nondecreasing on $I$.
\item For any $t_1 \in I$,
\begin{equation*}
\left.\frac{d}{dt}P_n[\log p_t]\right|_{t=t_1} = 0
\quad\Longleftrightarrow\quad
m_n^{(t_1)} = 0 \ \text{in } \HK^{(t_1)}.
\end{equation*}
\end{enumerate}
\end{lemma}

\Cref{thm::finitestep} guarantees that the ULFS--KDPE flow reaches the algorithmic stopping criterion in finite time. In particular, the empirical score $P_n D(p_t)$ cannot remain uniformly above the target tolerance $\delta_n$ throughout the interval $[0,T]$. This ensures that the ULFS--KDP
estimator is well defined and that the stopping rule is not merely heuristic. The proof is allocated to Appendix D, \Cref{append::convergence}.

%The key mechanism is a Lyapunov-type argument. Along the ULFS flow, $P_n[\log p_t]$ increases at rate $P_n D(p_t)$. If this rate stayed strictly above $\delta_n$, the empirical log-likelihood would grow linearly in $t$, forcing exponential inflation of the density near at least one observed point. The assumed local mass of the initial density then implies that the total mass of $p_t$ would exceed one at time $T$, contradicting normalization. Hence, the flow must cross the target accuracy level in finite time.

\begin{theorem}
\label{thm::finitestep}
There exists $t \in [0 ,T]$ such that $P_nD(p_t) \leq \delta_n$.
\end{theorem}

\subsection{Asymptotic Linearity and Efficiency}

Finally, we establish asymptotic linearity and efficiency of the ULFS--KDPE estimator. The central requirement is that the RKHS is rich enough to approximate the canonical gradient while remaining well behaved empirically. Assumption~\ref{ass::universal} formalizes this condition. Under \Cref{ass::universal}, Theorem~\ref{thm::efficiency} shows that solving the empirical score equation yields an asymptotically linear estimator with influence function $\phi^*_{P^*}$, and hence semiparametric efficiency. The proof is allocated to Appendix D, \Cref{append::efficiency}.

\begin{assumption}
\label{ass::universal}
There exists an event $\Omega^{'}$ with $P^*(\Omega^{'})=1$ such that for every $\omega\in\Omega^{'}$ there exists a sequence
$\{h_j(\omega)\}_{j\ge1}\subset \mathcal H_K^{(t_n(\omega))}(\omega)$ satisfying:
\begin{enumerate}
\item[(i)]
%$L^2(P_{t_n})$ norm
$\|h_j(\omega)-\phi^*_{P_{t_n(\omega)}}(\omega)\|_{L^2(P^*)}\to 0$ as $j\to\infty$;
\item[(ii)]
$\sup_{j\ge1}\|h_j(\omega)\|_{\mathcal H_K^{(t_n(\omega))}}<\infty$;
\item[(iii)]
there exists $j_0(\omega)<\infty$ such that the class
$\{h_j(\omega): w\in\Omega^{'}, j\ge j_0(\omega)\}$ is $P^*$-Donsker.
\end{enumerate}
\end{assumption}

We pick $\delta_n=o(n^{-2})$, and let $t_n\in [0,T]$ such that $P_nD(p_{t_n})\leq \delta_n$. Then, we have the following result.

\begin{theorem}\label{thm::efficiency}
Let $\Psi: \mathcal{M}\to \R$ be a pathwise-differentiable functional of the distribution $P$ with canonical gradient $\phi^*_{P}\in L_0^2(P)$. Let $t_n\in [0,T]$, which satisfies $P_nD(p_{t_n})\leq \delta_n$, and $\hat P_n:=P_{t_n}$. Under Assumptions \ref{ass::emprocess}, \ref{ass::2orderrem}, \ref{ass::kernelbound}, \ref{ass::universal}, $P_n\phi^*_{P_{t_n}}=o_{P^*}(n^{-1/2})$ and the ULFS--KDPE estimator satisfies 
\begin{equation*}
\Psi(\hat P_n)-\Psi(P^*)=P_n\phi^*_{P^*}+o_{P^*}(n^{-1/2}).
\end{equation*}
\end{theorem}

\section{Simulation Results}\label{sec::Simulation}

We conduct a simulation study to illustrate implementation of the ULFS–KDPE estimator, verify the theoretical results, and assess finite-sample performance across two data-generating processes (DGPs). We evaluate bias, variance, stability of the debiasing procedure, and compare five stopping criteria described in \Cref{subsubsec::stopping}. All experiments use 500 Monte Carlo replications with sample size $n=300$, a Gaussian mean-zero kernel, step size $\Delta=0.01$, and a maximum of 100 iterations. Density stabilization is used for estimation with lower bound $c=10^{-3}$ and tolerance $\lambda=10^{-8}$.

\subsection{Data Generating Processes and Target Parameters}\label{sub:DGP}

\subsubsection{DGP 1: Observational Study with Binary Outcome}

We define $\mathcal O = \mathcal{X}\times\mathcal{A}\times\mathcal{Y}$ with baseline covariates $X\in \mathcal X \equiv [0,1]$, binary treatment $A\in\mathcal A \equiv \{0,1\}$, and a binary outcome $Y \in \mathcal Y \equiv\{0,1\}$. The true DG is given below:
\begin{align*}
    X\sim \text{Unif}(0,1), 
    \quad A|X \sim \text{Bern}(0.5 + \frac{1}{3} \text{sin} (50X/\pi))\\
    Y|A,X \sim \text{Bern}(0.4 + A(X - 0.3)2 + \frac{1}{4} \text{sin}(40X/\pi))
\end{align*}

\subsubsection{DGP 2: Observational Study with a Positivity Issue and Binary Outcome}

We define $\mathcal O = \mathcal{X}\times\mathcal{A}\times\mathcal{Y}$ with baseline covariates $X\in \mathcal X \equiv [-1,1]$, binary treatment $A\in\mathcal A \equiv \{0,1\}$, and a binary outcome $Y \in \mathcal Y \equiv\{0,1\}$. The true DG is given below:
\begin{align*}
    &X\sim 0.9\cdot\text{Unif}(-1,1) + 0.1 \cdot \text{Unif}(-2,2), \\
    &A|X \sim \text{Bern}(\text{expit}(4X)),\\
    &Y|A,X \sim \text{Bern}(\text{expit}(-0.5+A+0.5X)).
\end{align*}
We note that the proposed DGP has a positivity issue. In particular, $X$ close to  $\pm2$ leads to $e^*$ very close to $0$ or $1$.

\subsubsection{Target Parameters}
We focus on several canonical causal parameters. We define the mean potential outcome under treatment level $a \in \{0,1\}$ as
\begin{align*}
\mu_a(P)
:= \E_P\bigl[Y^a\bigr]
&= P_X\!\left[\E_P\bigl[Y \mid A=a,X\bigr]\right]\\
&= \int \bar Q_P(a,x)\,dP_X(x),    
\end{align*}where $\bar Q_P(a,x):= \E_P[Y\mid A=a,X=x]$. We consider the following  parameters: 
\begin{enumerate}
    \item the average treatment effect (ATE), $\psi_{\mathrm{ATE}}(P) := \mu_1(P)-\mu_0(P)$;
    \item  the risk ratio (RR),
$\psi_{\mathrm{RR}}(P) := \mu_1(P)/\mu_0(P)$ and 
\item  the odds ratio (OR), $\psi_{\mathrm{OR}}(P) := \frac{\mu_1(P)/(1-\mu_1(P))}{\mu_0(P)/(1-\mu_0(P))}$.
\end{enumerate}

\subsubsection{Simulation Set-up}

For all DGPs, we initialize the distribution of baseline covariate $X$ as $P_n (X)$, and treat the empirical distribution of $X$ as fixed along the targeting step.  The remaining conditional distributions are estimated using the Super Learner framework \cite{sl2007}, as implemented in the \texttt{sl3} package~\cite{coyle2021sl3} in \textsf{R}. The candidate learner library includes the sample mean, generalized linear models (GLM), random forests~\cite{breiman2001random}, and gradient boosting via XGBoost~\cite{chen2016xgboost}. %{\color{red}{CITE all the learners}}

For both DGPs, we compare the proposed ULFS-KDPE with the following methods: (i) the original iterative KDPE, which relies on the locally least favorable submodel \cite{cho2024kernel}; (ii) targeted maximum likelihood estimation (TMLE) \cite{vanderLaanRubin2006TMLE} and (iii) the one-step TMLE \cite{one_step_tmle}. TMLE and one-step TMLE were implemented using the \texttt{tmle3} package~\cite{coyle2021tmle3} in \textsf{R}. One-step TMLE was included only for the average treatment effect $\psi_{\mathrm{ATE}}$, as the current \texttt{tmle3} implementation does not support one-step estimation of the risk ratio $\psi_{\mathrm{RR}}$ or odds ratio $\psi_{\mathrm{OR}}$. Both TMLE and one-step TMLE explicitly use the EIF of the target parameter and are known to attain the semiparametric efficiency bound under standard regularity conditions.

\subsection{Empirical Performance}

Across all DGPs, the simulations strongly support the theoretical properties of ULFS–KDPE and demonstrate clear practical advantages. In a well-behaved observational setting with a binary outcome (DGP1), ULFS–KDPE exhibits behavior consistent with first-order efficiency, with the empirical distribution of $\psi_{\mathrm{ATE}}(\hat P_{\mathrm{ULFS\text{-}KDPE}})$ closely matching its asymptotic normal limit. Compared with TMLE and one-step TMLE, ULFS–KDPE achieves smaller RMSE across all DGPs in \Cref{table:summary}, indicating a favorable bias–variance tradeoff induced by the regularized flow. Compared with (iterative) KDPE, ULFS–KDPE exhibits substantially lower bias in almost all cases, while maintaining comparable variance. A key advantage is that a single ULFS–KDPE distribution can be used to estimate multiple pathwise differentiable parameters simultaneously. This is especially evident for nonlinear targets such as $\psi_{\mathrm{RR}}$ and $\psi_{\mathrm{OR}}$, where ULFS–KDPE yields lower bias and RMSE than TMLE, despite TMLE requiring separate targeting steps.

The advantages of ULFS–KDPE are particularly evident under positivity violations (DGP2). In this more challenging setting, ULFS–KDPE consistently outperforms TMLE and one-step TMLE across all target parameters, exhibiting markedly lower variance and closer agreement with the limiting distributions. For nonlinear targets such as $\psi_{\mathrm{RR}}$ and $\psi_{\mathrm{OR}}$, ULFS–KDPE achieves the smallest bias and RMSE, highlighting the stabilizing effect of the RKHS-restricted universal least favorable flow and density regularization in settings where EIF-based methods suffer from variance inflation.

In addition to its statistical performance, ULFS–KDPE exhibits markedly improved numerical stability compared to the original (iterative) KDPE. ULFS–KDPE converges within the iteration limit in substantially more simulations, reflecting the robustness of its micro-step updating scheme. This scheme discretizes a globally defined least favorable flow and avoids the overshooting behavior that can hinder convergence of locally defined KDPE updates. As shown in \Cref{table:iteration}, for both DGPs ULFS–KDPE converges in all simulations when the iteration limit is moderately increased to 150.

The behavior of the stopping criteria further reinforces the algorithmic stability of ULFS–KDPE. As shown in \Cref{table:stopsummary}, stopping rules aligned with the geometry of the ULFP flow are the most reliable across all DGPs. The density stabilization rule, in particular, exploits the monotonicity of the empirical log-likelihood and halts the procedure once additional updates no longer yield meaningful reductions in empirical bias. In contrast, stopping rules based on local score conditions or explicit EIF-solving are more sensitive to finite-sample instability, especially for nonlinear targets and under positivity violations. The variance-based stopping rule is conservative and tends to terminate late.

Overall, these results demonstrate that ULFS–KDPE achieves first-order efficiency in regular settings, provides substantial finite-sample variance reduction under challenging positivity regimes, simultaneously debiases multiple target parameters using a single distribution, and offers improved numerical stability when coupled with stopping rules intrinsic to the universal least favorable flow.

\begin{table}[H]
    \caption{Bias, Convergence with 100 Steps, Variance and Root Mean Squared Error (RMSE) for ATE, RR and OR with different estimators (ULFS-KDPE, KDPE, TMLE, One-step TMLE).}
    \label{table:summary}
    \begin{center}
    \begin{threeparttable}
    \resizebox{0.97\columnwidth}{!}{
        \begin{small}
            \resizebox{\columnwidth}{!}{
            \begin{tabular}{rclrrrr}
                \hline
                    &Parameter &Method &\#Cov.  & Bias($\times100$) & Var\tnote{a} & RMSE \\ 
                \hline
                    & &ULFS-KDPE &475 & -0.824 & 0.0561 & 0.0567 \\ 
                DGP1 & $\psi_{ATE}$ &KDPE &343 & -1.438 & 0.0546 & 0.0565 \\
                    & &TMLE &500 & -0.004 & 0.0618 & 0.0618 \\
                    & &One-step TMLE &500 & -0.247 & 0.0605 & 0.0606 \\
                \cmidrule(lr){2-7}
                    & &ULFS-KDPE &475 & -1.082 & 0.1724 & 0.1728 \\ 
                    & $\psi_{RR}$ &KDPE &343 & -3.213 & 0.1584 & 0.1616 \\
                    & &TMLE &500 &1.694 & 0.1919 & 0.1928 \\
                \cmidrule(lr){2-7}
                    & &ULFS-KDPE &475 & -0.499 & 0.3951 & 0.3954 \\ 
                    & $\psi_{OR}$ &KDPE &343 & -5.010 & 0.3683 & 0.3717 \\
                    & &TMLE &500 & 6.193 & 0.4464 & 0.4510 \\
                \hline
                    & &ULFS-KDPE &500 & -0.779 & 0.0772 & 0.0777 \\ 
                DGP2 & $\psi_{ATE}$ &KDPE &389 & 0.034 & 0.0750 & 0.0751 \\
                    & &TMLE &500 & 0.128& 0.1203 & 0.1204 \\
                    & &One-step TMLE &500 & 0.009 & 0.1176 & 0.1177 \\
                \cmidrule(lr){2-7}
                    & &ULFS-KDPE &500 & 0.427 & 0.2777 & 0.2779 \\ 
                    & $\psi_{RR}$ &KDPE &389 & 3.123 & 0.2766 & 0.2784 \\
                    & &TMLE &500 & 8.663 & 0.4599 & 0.4683 \\
                \cmidrule(lr){2-7}
                    & &ULFS-KDPE &500 & 8.729 & 0.9505 & 0.9552 \\ 
                    & $\psi_{OR}$ &KDPE &389 & 17.575 & 0.9621 & 0.9784 \\
                    & &TMLE &500 & 49.752 & 1.6948 & 1.7674 \\
                \hline
            \end{tabular}
            }
        \end{small}
        }
        \begin{tablenotes}[flushleft]
            \scriptsize
            \item[a] Var denotes the Monte Carlo variance of the estimator across $B$ simulated datasets, computed as 
$\frac{1}{B}\sum_{b=1}^B (\hat{\psi}_b - \bar{\psi})^2$, where 
$\bar{\psi} = \frac{1}{B}\sum_{b=1}^B \hat{\psi}_b$.
        \end{tablenotes}
        \end{threeparttable}
    \end{center}
  \vskip -0.1in
\end{table}

\begin{figure}[H]
    \begin{center}
\centerline{\includegraphics[width=\columnwidth]{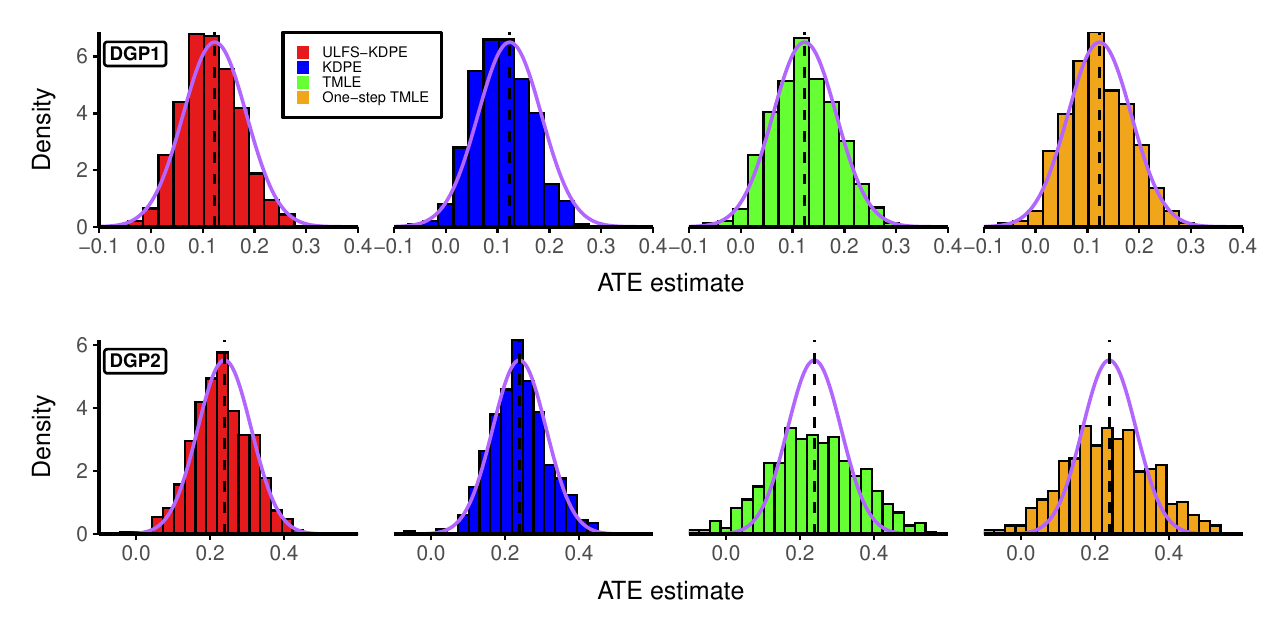}}
        \caption{Finite-sample behavior of the different ATE estimators (ULFS-KDPE, KDPE, TMLE, One-step TMLE). True asymptotic distribution is depicted in purple.}
        \label{fig:ATEbinary}
    \end{center}
\end{figure}

\begin{table}[H]
    \caption{Bias, Convergence with 200 Steps, Var, RMSE for ATE, RR and OR under different stopping criteria for ULFS-KDPE.}
    \label{table:stopsummary}
    \begin{center}
    \begin{threeparttable}
    \resizebox{0.93\columnwidth}{!}{
        \begin{small}
            \begin{tabular}{rclrrrr}
                \hline
                  &Parameter &Stopping Criterion & \#Cov.  & Bias ($\times100$) & Var\tnote{a} & RMSE \\ 
                \hline
                   & &Density stabilization &500 &-0.840 & 0.0553 & 0.0560 \\ 
                DGP1 & $\psi_{ATE}$ &Empirical score stabilization &493 &-0.806 &0.0566 &0.0572 \\
                   & &Vanishing update direction &500 &-1.357 &0.0547 &0.0564 \\
                   & &Variance-dominated updates &500 &-1.573 &0.0550 &0.0573 \\
                   & &EIF approximately solved &447 &-1.156 &0.0561 &0.0573 \\
                \cmidrule(lr){2-7}
                   & &Density stabilization &500 &-1.142 &0.1700 &0.1705 \\ 
                   & $\psi_{RR}$ &Empirical score stabilization &493 &-0.985 &0.1738 &0.1742 \\
                   & &Vanishing update direction &500 &-2.690 &0.1647 &0.1670 \\
                   & &Variance-dominated updates &500 &-3.274 & 0.1651 & 0.1684 \\
                   & &EIF approximately solved &447 &-2.070 &0.1698 &0.1712 \\
                \cmidrule(lr){2-7}
                   & &Density stabilization &500 &-0.727 &0.3893 &0.3896 \\ 
                   & $\psi_{OR}$ &Empirical score stabilization &493 &-0.274 &0.4002 &0.4005 \\
                   & &Vanishing update direction &500 &-4.301 &0.3759 &0.3786 \\
                   & &Variance-dominated updates &500 &-5.648 & 0.3758 & 0.3803 \\
                   & &EIF approximately solved &447 &-2.750 &0.3897 &0.3909 \\
                \hline
                   & &Density stabilization &500 &-0.779 & 0.0772 & 0.0777 \\ 
                DGP2 & $\psi_{ATE}$ &Empirical score stabilization &499 &-0.511 &0.0770 &0.0772 \\
                   & &Vanishing update direction &497 &-0.981 &0.0779 &0.0785  \\
                   & &Variance-dominated updates &500 &-1.240 & 0.0780 & 0.0790 \\
                   & &EIF approximately solved &306 &-1.269 &0.0786 &0.0797 \\
                \cmidrule(lr){2-7}
                   & &Density stabilization &500 &0.427 & 0.2777 & 0.2779 \\ 
                   & $\psi_{RR}$ &Empirical score stabilization &499  &1.374 &0.2788 &0.2794 \\
                   & &Vanishing update direction &497 &-0.238 &0.2784 &0.2786\\
                   & &Variance-dominated updates &500 &-1.114 & 0.2779 & 0.2783 \\
                   & &EIF approximately solved &306 &-1.069 &0.2797 &0.2799 \\
                \cmidrule(lr){2-7}
                   & &Density stabilization &500 & 8.729 & 0.9505 & 0.9552 \\ 
                   & $\psi_{OR}$ &Empirical score stabilization &499 &11.805 &0.9572 &0.9650 \\
                   & &Vanishing update direction &497  &6.629 &0.9504 &0.9533 \\
                   & &Variance-dominated updates &500 & 3.712 & 0.9441 & 0.9455 \\
                   & &EIF approximately solved &306  &3.620 &0.9419 &0.9427 \\
                \hline
            \end{tabular}
        \end{small}
        }
        \begin{tablenotes}[flushleft]
            \scriptsize
            \item[a] $\frac{1}{B}\sum_{b=1}^B (\hat{\psi}_b - \bar{\psi})^2$, where 
$\bar{\psi} = \frac{1}{B}\sum_{b=1}^B \hat{\psi}_b$.
        \end{tablenotes}
    \end{threeparttable}
    \end{center}
  \vskip -0.1in
\end{table}

\begin{figure}[H]
    \centering
 \includegraphics[width=0.95\columnwidth]{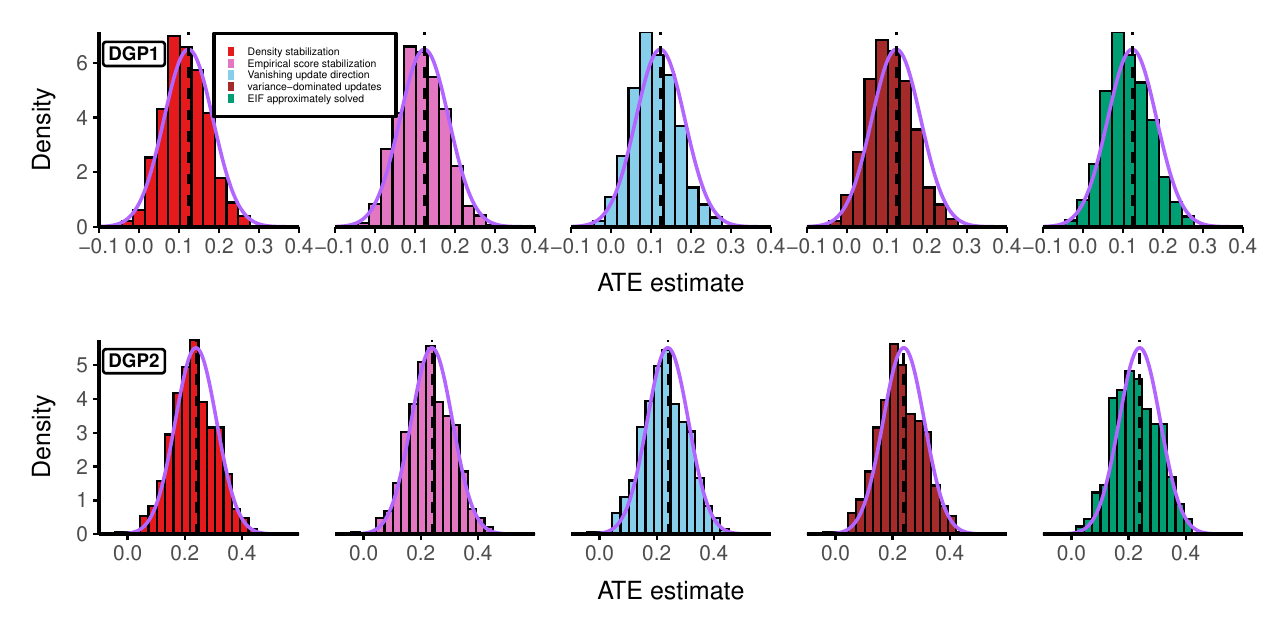}
    \caption{ULFS-KDPE under different stopping criteria. True asymptotic distribution is depicted in purple.}
    \label{fig:stop}
\end{figure}

%\begin{figure}[H]
%    \centering \includegraphics[width=\columnwidth]{RROR.pdf}
%    \caption{Simulated distribution of $\hat{\psi}_{RR}$ and $\hat{\psi}_{OR}$.  First row corresponds to DGP1, and second row to DGP2. Red line denotes true value of the target parameter.}
%    \label{fig:RROR}
%\end{figure}

\begin{table}[H]
    \caption{Convergence and RMSE for ATE, RR and OR with ULFS-KDPE under different iteration limits (\# of steps allowed).}
    \label{table:iteration}
    \begin{center}
        \begin{small}
            \begin{tabular}{rrlrrrr}
              \hline
              &Iteration Limit &Stopping Criterion &\#Cov. & $\psi_{\text{ATE}}$ & $\psi_{\text{RR}}$ & $\psi_{\text{OR}}$ \\ 
              \hline
               & &Density plateau &475 & 0.0567 & 0.1728 & 0.3975 \\ 
            DGP1 &100  &Score plateau &404 & 0.0573 & 0.1676 & 0.3792 \\ 
               & &Vanishing Direction &498 & 0.0563 & 0.1669 & 0.3784 \\ 
               & &Variance dominate &500 & 0.0567 & 0.1684 & 0.3803 \\ 
               & &EIF sovled &407 & 0.0547 & 0.1721 & 0.3943 \\ 
               \cmidrule(lr){2-7}
               & &Density plateau &500 & 0.0560 & 0.1702 & 0.3896 \\ 
               &150  &Score plateau &472 & 0.0569 & 0.1691 & 0.3900 \\ 
               & &Vanishing Direction &500 & 0.0564 & 0.1650 & 0.3786 \\ 
               & &Variance dominate &500 & 0.0573 & 0.1654 & 0.3803 \\ 
               & &EIF sovled &435 & 0.0576 & 0.1711 & 0.3930 \\ 
               \cmidrule(lr){2-7}
               & &Density plateau &500 & 0.0560 & 0.1702 & 0.3896 \\ 
               &200  &Score plateau &493 & 0.0572 & 0.1741 & 0.4005 \\ 
               & &Vanishing Direction &500 & 0.0564 & 0.1650 & 0.3786 \\ 
               & &Variance dominate &500 & 0.0573 & 0.1654 & 0.3803 \\ 
               & &EIF sovled &447 & 0.0573 & 0.1701 & 0.3909 \\ 
               \hline
               & &Density plateau &500 & 0.0777 & 0.2779 & 0.9552 \\ 
            DGP2 &100  &Score plateau &197 & 0.0838 & 0.2942 & 1.0551 \\ 
               & &Vanishing Direction &484 & 0.0786 & 0.2775 & 0.9501 \\ 
               & &Variance dominate &500 & 0.0790 & 0.2783 & 0.9455 \\ 
               & &EIF sovled &251 & 0.0818 & 0.2844 & 0.9553 \\ 
               \cmidrule(lr){2-7}
               & &Density plateau &500 & 0.0777 & 0.2779 & 0.9552 \\ 
               &150  &Score plateau &415 & 0.0782 & 0.2757 & 0.9476 \\ 
               & &Vanishing Direction &492 & 0.0783 & 0.2769 & 0.9458 \\ 
               & &Variance dominate &500 & 0.0790 & 0.2783 & 0.9455 \\ 
               & &EIF sovled &286 & 0.0806 & 0.2820 & 0.9509 \\ 
               \cmidrule(lr){2-7}
               & &Density plateau &500 & 0.0777 & 0.2779 & 0.9552 \\ 
               &200  &Score plateau &499 & 0.0772 & 0.2794 & 0.9650 \\ 
               & &Vanishing Direction &497 & 0.0785 & 0.2786 & 0.9533 \\ 
               & &Variance dominate &500 & 0.0790 & 0.2783 & 0.9455 \\ 
               & &EIF sovled &306 & 0.0797 & 0.2799 & 0.9427 \\ 
               \hline
            \end{tabular}
        \end{small}
    \end{center}
\end{table}

\section{Discussion and Future Work}

This work develops a kernel debiased plug-in estimator based on the universal least favorable submodel (ULFS--KDPE), providing a unified and theoretically grounded framework for efficient semiparametric estimation that is both influence-function-free and computationally tractable. By embedding debiasing in an RKHS  and defining the update through a universal least favorable flow, the proposed method departs from classical locally targeted approaches and instead enforces least favorability globally along a single distributional path.

From a theoretical perspective, the main contribution lies in placing ULFS--KDPE on a rigorous functional-analytic foundation. We formulate the universal least favorable update as a nonlinear ODE on densities and establish existence, uniqueness, stability, and finite-time convergence of its solutions in appropriate Hölder spaces. These results guarantee that the proposed flow is well posed, preserves positivity and normalization, and reaches a point where the empirical RKHS score is sufficiently small in finite time. Importantly, the construction yields a plug-in estimator that is regular, asymptotically linear, and semiparametrically efficient under standard conditions, without requiring explicit derivation of the efficient influence function. Moreover, efficiency is attained simultaneously for all pathwise differentiable parameters whose canonical gradients lie in the $L^2(P_0)$-closure of the RKHS, including multivariate targets.

From a practical standpoint, ULFS--KDPE inherits advantages from both universal least favorable submodels and kernel-based debiasing. The global least favorability of the path avoids convergence pathologies that can arise from iterative local targeting, while the RKHS representation reduces the update to finite-dimensional computations involving kernel evaluations at the observed data points. The Lyapunov structure of the flow provides a geometric interpretation of the update as a stabilized gradient flow that monotonically increases the empirical log-likelihood and approaches stationarity as the empirical score vanishes. Simulation results support these theoretical properties and demonstrate improved finite-sample stability and accuracy relative to locally targeted methods, particularly in challenging regimes such as limited overlap.

Several directions for future research emerge from this work. First, a deeper theoretical analysis of stopping criteria is needed. While the current criteria are motivated by the Lyapunov structure and empirical score equations, formal results characterizing their impact on asymptotic linearity, efficiency, and finite-sample bias would strengthen the inferential guarantees of ULFS--KDPE. In particular, developing data-adaptive stopping rules with explicit rates and inference-valid guarantees is an important next step. Second, further investigation of discretization schemes is warranted. Our implementation relies on an explicit Euler discretization of the ULFS ODE, which is simple and stable but may not be optimal. Understanding the interaction between discretization error, kernel smoothness, and convergence rates remains an open problem.

Extending the framework to higher-order inference is another promising direction. While ULFS--KDPE achieves first-order semiparametric efficiency, the universal RKHS-based score equations suggest the potential to capture higher-order components of the influence function \cite{pimentel2025scorepreservingtargetedmaximumlikelihood}. Developing a systematic theory for second-order or higher-order expansions along the universal least favorable flow could lead to improved inference in moderate samples and provide new insights into the geometry of efficient estimation. Finally, scalability and approximation deserve further attention. Random feature approximations, low-rank kernel methods, or localized kernels may allow ULFS--KDPE to scale to larger datasets while preserving its theoretical properties. Exploring these directions would broaden the practical applicability of the method and further bridge the gap between strong semiparametric theory and modern data-intensive applications.

\section{Acknowledgment}
\noindent
We thank Patrick Lopatto for many comments and discussions on the topic.

\newpage

\section{Appendix A}\label{append::A}

\begin{proposition}[Expectation as an RKHS functional]
Let $P$ be any probability measure on $\O$, and let $\HK$ be the RKHS associated with the Gaussian kernel $K$. Then the kernel mean embedding $m_P:=\int k_o\,dP(o)$ exists in $\HK$. Moreover, for every $f\in\HK$,
\[
P[f]=\langle f,m_P\rangle_{\HK}.
\]
Consequently, the map $f\mapsto P[f]$ is a continuous linear functional on $\HK$ satisfying
\[
|P[f]|\le \|f\|_{\HK}.
\]
\end{proposition}

\begin{proof}
Since $K(o,o)\le 1$ for all $o\in\O$, we have
\[
\|k_o\|_{\HK}^2=\langle k_o,k_o\rangle_{\HK}=K(o,o)\le 1,
\]
and therefore $\|k_o\|_{\HK}\le 1$ for all $o\in\O$. It follows that $\int \|k_o\|_{\HK}\,dP(o)\le 1<\infty$. Hence the $\HK$-valued Bochner integral $m_P:=\int k_o\,dP(o)$ is well defined.

Let $f\in\HK$. By the reproducing property,
\[
f(o)=\langle f,k_o\rangle_{\HK}
\qquad \text{for all }o\in\O.
\]
Therefore,
\[
P[f]=\int f(o)\,dP(o)=\int \langle f,k_o\rangle_{\HK}\,dP(o).
\]
Since the map $g\mapsto \langle f,g\rangle_{\HK}$ is a bounded linear functional on $\HK$, it commutes with the Bochner integral, yielding
\[
\int \langle f,k_o\rangle_{\HK}\,dP(o)
=
\left\langle f,\int k_o\,dP(o)\right\rangle_{\HK}
=
\langle f,m_P\rangle_{\HK}.
\]
Thus
\[
P[f]=\langle f,m_P\rangle_{\HK}.
\]
Finally, by the Cauchy--Schwarz inequality,
\[
|P[f]|=|\langle f,m_P\rangle_{\HK}|
\le \|f\|_{\HK}\,\|m_P\|_{\HK}.
\]
Moreover,
\[
\|m_P\|_{\HK}
=
\left\|\int k_o\,dP(o)\right\|_{\HK}
\le
\int \|k_o\|_{\HK}\,dP(o)
\le 1.
\]
Hence
\[
|P[f]|\le \|f\|_{\HK}.
\]
This proves that $f\mapsto P[f]$ is a continuous linear functional on $\HK$.
\end{proof}

\begin{proposition}[Mean-zero RKHS as a closed subspace]
For each $t\in I$, the space $\HKt$ is a closed linear subspace of $\HK$. Moreover,
\[
\HKt
=
\{f\in\HK:\langle f,m_{\Pt}\rangle_{\HK}=0\}
=
(\operatorname{span}\{m_{\Pt}\})^\perp.
\]
In particular, $\HKt$ is a Hilbert space under the inner product inherited from $\HK$.
\end{proposition}

\begin{proof}
By Proposition~\ref{prop::RKHS_expectation}, for every $f\in\HK$ we have
\[
\Pt[f]=\langle f,m_{\Pt}\rangle_{\HK}.
\]
Hence the constraint $\Pt[f]=0$ is equivalent to orthogonality to $m_{\Pt}$. Since the map $f\mapsto \langle f,m_{\Pt}\rangle_{\HK}$ is a continuous linear functional on $\HK$, its null space
\[
\{f\in\HK:\langle f,m_{\Pt}\rangle_{\HK}=0\}
\]
is a closed linear subspace of $\HK$. This proves the claim.
\end{proof}

\begin{proposition}[Orthogonal projection onto the mean-zero RKHS]
If $m_{\Pt}\neq 0$, the orthogonal projection $\Pi_t:\HK\to\HKt$ is given by
\[
\Pi_t f
=
f-\frac{\langle f,m_{\Pt}\rangle_{\HK}}{\|m_{\Pt}\|_{\HK}^2}\,m_{\Pt}
=
f-\frac{\Pt[f]}{\|m_{\Pt}\|_{\HK}^2}\,m_{\Pt},
\qquad f\in\HK.
\]
\end{proposition}

\begin{proof}
Since $\HKt=(\operatorname{span}\{m_{\Pt}\})^\perp$, the orthogonal projection of $f\in\HK$ onto $\HKt$ is obtained by subtracting its component along $m_{\Pt}$. Writing
\[
\Pi_t f=f-c\,m_{\Pt}
\]
and imposing $\Pi_t f\in\HKt$ yields
\[
0=\langle f-c\,m_{\Pt},m_{\Pt}\rangle_{\HK}
=
\langle f,m_{\Pt}\rangle_{\HK}-c\|m_{\Pt}\|_{\HK}^2.
\]
Solving for $c$ gives
\[
c=\frac{\langle f,m_{\Pt}\rangle_{\HK}}{\|m_{\Pt}\|_{\HK}^2},
\]
which yields the stated formula.
\end{proof}

\begin{proposition}[Empirical Riesz representer in the mean-zero RKHS]
There exists a unique element $m_n^{(t)}\in\HKt$ such that
\[
P_n[f]=\langle f,m_n^{(t)}\rangle_{\HKt}
\qquad\text{for all }f\in\HKt.
\]
Moreover,
\[
m_n^{(t)}=\frac1n\sum_{i=1}^n \kt_{O_i}.
\]
\end{proposition}

\begin{proof}
Linearity of $L_n$ follows immediately from linearity of summation. To verify boundedness, note that for any $f\in\HKt\subset\HK$ and any $o\in\O$, the reproducing property gives $f(o)=\langle f,k_o\rangle_{\HK}$. Since the Gaussian kernel satisfies $K(o,o)\le1$, we have $\|k_o\|_{\HK}\le1$, and therefore
\[
|f(o)|\le \|f\|_{\HK}.
\]
Applying this bound to each observation $O_i$ yields
\[
|P_n[f]|
\le
\frac1n\sum_{i=1}^n |f(O_i)|
\le
\|f\|_{\HK}.
\]
Hence $L_n$ is a bounded linear functional on the Hilbert space $\HKt$, and the Riesz representation theorem guarantees the existence and uniqueness of $m_n^{(t)}\in\HKt$ satisfying
\[
P_n[f]=\langle f,m_n^{(t)}\rangle_{\HKt}.
\]
The explicit expression follows from the reproducing property in $\HKt$.
\end{proof}

\section{Appendix B}\label{append::B}

\begin{algorithm}[t]
\caption{ULFS--KDPE (discretized RKHS-restricted flow on $\O$)}
\label{alg:kdpe_ulfs_generalO}
\begin{algorithmic}[1]
\REQUIRE Data $O_1,\ldots,O_n$; initial density estimate $\hat p_0$; kernel $K$ on $\O$; step size $\Delta > 0$; tolerances $\delta_p,\delta_s,\delta_\alpha,\delta_v,\delta_\ell$; maximum iterations $M$.
\ENSURE Final iterate $\hat P_T$ (density $\hat p_T$) and plug-in estimate $\hat\psi_T := \Psi(\hat P_T)$.

\STATE Set $t\leftarrow 0$ and $\hat p_t\leftarrow \hat p_0$.
\STATE Initialize $\Delta_t^{(p)}\leftarrow +\infty$, $\Delta_t^{(\ell)}\leftarrow +\infty$, $s_t\leftarrow +\infty$.

\FOR{$m=0,1,\ldots,M-1$}

    \STATE Construct the current iterate $\hat P_t$ (with density $\hat p_t$).

    \STATE Center the kernel at $\hat P_t$ to obtain the mean-zero RKHS kernel $\Kt$ on $\O$ (equivalently, via the mean embedding $m_{\hat P_t}:=\int k_o\,d\hat P_t(o)$).
    \STATE Form the centered Gram matrix $G^{(t)}:=\bigl[\Kt(O_i,O_j)\bigr]_{i,j=1}^n$.

    \STATE Compute $\valpha^{(t)}\leftarrow \frac{1}{n}G^{(t)}\one$, so $\alpha_j^{(t)}=\P_n[\kt_{O_j}]=m_n^{(t)}(O_j)$.

    \STATE Define the RKHS direction function $D_t(o):=D(\hat p_t)(o):=\frac{1}{n}\sum_{j=1}^n \alpha_j^{(t)}\,\kt_{O_j}(o)$
    and its evaluations on the sample $\textbf d^{(t)}\leftarrow \bigl(D_t(O_i)\bigr)_{i=1}^n=\frac{1}{n}G^{(t)}\valpha^{(t)}$.

    \STATE Compute the Lyapunov score $s_t\leftarrow \P_n[D_t]=\frac{1}{n}\sum_{j=1}^n\bigl(\alpha_j^{(t)}\bigr)^2=\frac{1}{n}\|\valpha^{(t)}\|_2^2$.

    \IF{$m\ge 1$}
        \STATE Compute plateau diagnostics:
        $\Delta_t^{(p)}\leftarrow \P_n\!\left[\bigl(\log\hat p_t(O)-\log\hat p_{t-\Delta}(O)\bigr)^2\right]$,
        \STATE \hspace{1.45em}
        $\Delta_t^{(\ell)}\leftarrow \Bigl|\P_n\!\left[\log\hat p_t(O)-\log\hat p_{t-\Delta}(O)\right]\Bigr|$.
    \ENDIF

    \IF{stopping criteria in \S\ref{subsubsec::stopping} are satisfied}
        \STATE \textbf{break}
    \ENDIF

    \STATE Update the density by an explicit Euler step on the log-density:
    \STATE \hspace{1.45em}
    $\hat p_{t+\Delta}(o)\leftarrow \hat p_t(o)\exp\!\bigl(\Delta\,D_t(o)\bigr)$.

    \STATE Normalize:
    $\hat p_{t+\Delta}(o)\leftarrow \hat p_{t+\Delta}(o)\Big/\int \hat p_{t+\Delta}(u)\,d\mu(u)$.

    \STATE Set $t\leftarrow t+\Delta$ and overwrite $\hat p_t\leftarrow \hat p_{t+\Delta}$.
\ENDFOR

\STATE Set $T\leftarrow t$ and output $\hat P_T$ with density $\hat p_T:=\hat p_t$ and $\hat\psi_T:=\Psi(\hat P_T)$.
\end{algorithmic}
\end{algorithm}

\section{Appendix C}\label{append::C}

\subsection{Product Estimates and Banach Algebra Structure of Hölder space}
\label{append::D01}

To establish well-posedness of the nonlinear ODE in Equation~\eqref{eq::newode} in a
Banach space setting, it is essential to verify that the Hölder space $C^{1,\alpha}(\mathcal{O})$ is stable under multiplication and that the
corresponding norms can be controlled quantitatively. The first result below, Lemma~\ref{lemma:product}, provides a basic Hölder seminorm estimate for products of functions, showing that multiplication does not amplify oscillations beyond those already present in the individual factors. Building on this, Lemma~\ref{lem:Calg} establishes that $C^{1,\alpha}(\mathcal{O})$ is a Banach algebra: the product of two $C^{1,\alpha}$ functions remains in $C^{1,\alpha}$, with an explicit norm bound depending only on the domain and the Hölder exponent. Together, these
estimates form the analytic backbone needed to control the nonlinear vector
field $p \mapsto pD(p)$ and to apply standard existence and uniqueness results
for ordinary differential equations in Banach spaces.

\begin{lemma}
\label{lemma:product}
For $g,h \in C^{0,\alpha}(\mathcal{O})$, $$[gh]_{\alpha} \leq \|g\|_{\infty} [h]_{\alpha} + \|h\|_{\infty} [g]_{\alpha}$$
\end{lemma}

\begin{proof}
For any distinct $x,y\in [0,1]^d$ and fixed $b\in\{0,1\}^2$,
\[
g(x,b)h(x,b)-g(y,b)h(y,b)
= g(x,b)\big(h(x,b)-h(y,b)\big) + h(y,b)\big(g(x,b)-g(y,b)\big).
\]
For the slice $b$, taking absolute values and dividing by $\|x-y\|_2^\alpha$ gives
\[
\frac{|g(x)h(x)-g(y)h(y)|}{\|x-y\|_2^\alpha}
\le |g(x)|\frac{|h(x)-h(y)|}{\|x-y\|_2^\alpha}
+ |h(y)|\frac{|g(x)-g(y)|}{\|x-y\|_2^\alpha}.
\]
Now take the supremum over all $x\neq y \in [0,1]^d$, we have
\[
[gh]_{\alpha}
\le \|g\|_{\infty}\, [h]_{\alpha}
+\|h\|_{\infty}\, [g]_{\alpha}
\]
for any slice $b$.
\end{proof}

The next Lemma can be applied to any fixed slice $b\in\{0,1\}^2$.

\begin{lemma}\label{lem:Calg}
Let $\O\subset\mathbb R^d$ be convex and bounded, and let $\alpha\in(0,1]$. Then there exists a constant $C_{\mathrm{alg}}<\infty$ such that
\[
\|fg\|_{C^{1,\alpha}(\O)}
\le C_{\mathrm{alg}}\,
\|f\|_{C^{1,\alpha}(\O)}\,
\|g\|_{C^{1,\alpha}(\O)}
\qquad \forall f,g\in C^{1,\alpha}(\O).
\]
Moreover, one may take the explicit bound
\[
C_{\mathrm{alg}} = 1 + 2\,\mathrm{diam}(\O)^{\,1-\alpha},\quad
\mathrm{diam}(\O):=\sup_{x,y\in\O}\|x-y\|_2.
\]
In particular, for $\O=[0,1]^d$,
\[
\mathrm{diam}(\O)=\sqrt d
\quad\text{and thus}\quad
C_{\mathrm{alg}} = 1 + 2\,d^{\frac{1-\alpha}{2}}.
\]
\end{lemma}

\begin{proof}
Fix $f,g\in C^{1,\alpha}(\O)$, write
\[
A:=\|f\|_{\infty},B:=\|\nabla f\|_{\infty},C:=[\nabla f]_{\alpha},
A':=\|g\|_{\infty},B':=\|\nabla g\|_{\infty},C':=[\nabla g]_{\alpha}.
\]

By the definition of $\|\cdot\|_\infty$,
\[
\|fg\|_{\infty}\le \|f\|_{\infty}\|g\|_{\infty}=AA'.
\]

Using the product rule $\nabla(fg)=f\,\nabla g+g\,\nabla f$,
\[
\|\nabla(fg)\|_{\infty}
\le \|f\|_{\infty}\|\nabla g\|_{\infty}
   +\|g\|_{\infty}\|\nabla f\|_{\infty}
=AB'+A'B.
\]

For $x,y\in [0,1]^d$, $x\neq y$, we write
\[
\begin{aligned}
\nabla(fg)(x)-\nabla(fg)(y)
&= f(x)\nabla g(x)+g(x)\nabla f(x)-f(y)\nabla g(y)-g(y)\nabla f(y)\\
&= f(x)\big(\nabla g(x)-\nabla g(y)\big)
   +\big(f(x)-f(y)\big)\nabla g(y)\\
&\quad + g(x)\big(\nabla f(x)-\nabla f(y)\big)
   +\big(g(x)-g(y)\big)\nabla f(y).
\end{aligned}
\]
Taking Euclidean norms and applying the triangle inequality gives
\[
\begin{aligned}
\|\nabla(fg)(x)-\nabla(fg)(y)\|_2
&\le \|f\|_{\infty}\,\|\nabla g(x)-\nabla g(y)\|_2
     +|f(x)-f(y)|\,\|\nabla g\|_{\infty}\\
&\quad + \|g\|_{\infty}\,\|\nabla f(x)-\nabla f(y)\|_2
     +|g(x)-g(y)|\,\|\nabla f\|_{\infty}.
\end{aligned}
\]
Divide by $\|x-y\|_2^\alpha$ and take the supremum over $x\neq y$:
\[
[\nabla(fg)]_{\alpha}
\le A\,C' + A'\,C + \Big(\sup_{x\neq y}\frac{|f(x)-f(y)|}{\|x-y\|_2^\alpha}\Big)B'
                    + \Big(\sup_{x\neq y}\frac{|g(x)-g(y)|}{\|x-y\|_2^\alpha}\Big)B.
\]
It remains to control the Hölder seminorms of $f$ and $g$ by their gradients.
By the mean value theorem, for any $x,y\in [0,1]^d$,
\[
|f(x)-f(y)| \le \|\nabla f\|_{\infty}\,\|x-y\|_2 = B\,\|x-y\|_2,
\]
hence
\[
\frac{|f(x)-f(y)|}{\|x-y\|_2^\alpha}\le B\,\|x-y\|_2^{1-\alpha}
\le B\,\mathrm{diam}(\O)^{1-\alpha}.
\]
Taking the supremum over $x\neq y$ yields
\[
[f]_{\alpha} \le  \mathrm{diam}(\O)^{1-\alpha}\,B.
\]
Similarly, $[g]_{\alpha}\le \mathrm{diam}(\O)^{1-\alpha}\,B'$.
Plugging these into the previous display gives
\[
[\nabla(fg)]_{\alpha}
\le A\,C' + A'\,C + 2\,\mathrm{diam}(\O)^{1-\alpha}\,B\,B'.
\]

Summing the bounds above, we obtain
\[
\|fg\|_{C^{1,\alpha}(\O)}
\le AA' + (AB'+A'B) + (AC'+A'C) + 2\,\mathrm{diam}(\O)^{1-\alpha}\,BB'.
\]
Now observe that each term on the right-hand side is bounded by
\[
(1+2\,\mathrm{diam}(\O)^{1-\alpha})(A+B+C)(A'+B'+C'),
\]
since $AA',\,AB',\,A'B,\,AC',\,A'C \le (A+B+C)(A'+B'+C')$ and
$2\,\mathrm{diam}(\O)^{1-\alpha}BB' \le (1+2\,\mathrm{diam}(\O)^{1-\alpha})(A+B+C)(A'+B'+C')$.
Therefore,
\[
\|fg\|_{C^{1,\alpha}(\O)}
\le \bigl(1+2\,\mathrm{diam}(\O)^{1-\alpha}\bigr)
\|f\|_{C^{1,\alpha}(\O)}\,\|g\|_{C^{1,\alpha}(\O)}.
\]
This proves the claim with $C_{\mathrm{alg}}=1+2\,\mathrm{diam}(\O)^{1-\alpha}$.
Finally, when $\O=[0,1]^d$, the Euclidean diameter is
\[
\mathrm{diam}(\O)
=\sup_{x,y\in[0,1]^d}\|x-y\|_2
=\| (1,\dots,1)-(0,\dots,0)\|_2
=\sqrt d,
\]
and hence $C_{\mathrm{alg}} = 1 + 2\,d^{(1-\alpha)/2}$.
\end{proof}

\section{Appendix D}\label{append::D}

Recall that the observation space is $\mathcal O = [0,1]^d\times\{0,1\}\times\{0,1\}$ and write $O_k=(X_k,A_k,Y_k)$. For $(a,y)\in\{0,1\}^2$ define the slice $\mathcal O_{a,y}:=[0,1]^d\times\{a\}\times\{y\}$. For any $o=(x,A_k,Y_k) \in \mathcal O_{A_k,Y_k}$, the Euclidean distance to $O_k$ along the continuous coordinates is $\|o-O_k\|_2:=\|x-X_k\|_2$, which is well-defined on the slice $\mathcal O_{A_k,Y_k}$. Fix $\delta_n > 0$. Let $C$ and $M_{1,\alpha}$ be the constants from \Cref{ass::kernelbound} and define the radius $w(\delta_n) := \frac{\delta_n}{2 C M_{1,\alpha}}$. For each $k=1,\ldots,n$, define the slice-restricted ball
\[
B(O_k,w(\delta_n)) := \bigl\{o\in\mathcal O_{A_k,Y_k}:\ \|o-O_k\|_2<w(\delta_n) = \frac{\delta_n}{2 C M_{1,\alpha}} \bigr\}.
\]
Assume that $p_0$ assigns positive mass to each neighborhood $B(O_k,w(\delta_n))$,
%i.e. $\int_{B(O_k,w(\delta_n))} p_0(o)\,d\mu(o)>0$, $k=1,\ldots,n$,  
and set
\[
\mathfrak J(\delta_n) := \max_{1\le k\le n} \left(\int_{B(O_k,w(\delta_n))} p_0(o)\,d\mu(o)\right)^{-1}.
\]
To avoid a boundary/equality case in the contradiction argument, we introduce a fixed slack factor (any constant $>1$ suffices) and define $T:=\frac{2}{\delta_n}\log\!\bigl(2\,\mathfrak J(\delta_n)\bigr)$.
%NOTE: Originally, no fixed slack factor. Makes contradiction strict (this fixes the equality edge case in Theorem \ref{thm:finitestep}).
%\[
%T:=\frac{2}{\delta_n}\log\!\bigl(\,\mathfrak J(\delta_n)\bigr).
%\]
Let $M_0 := \|p_0\|_{C^{1,\alpha}(\mathcal O)}$ and define
\[
M := M_0\exp\!\bigl(C_{\mathrm{alg}} C M_{1,\alpha}T\bigr) = M_0\bigl(2\,\mathfrak J(\delta_n)\bigr)^{\frac{2C_{\mathrm{alg}}CM_{1,\alpha}}{\delta_n}}, 
\]
where $C_{\mathrm{alg}}$ is a constant in the Banach algebra property of $C^{1, \alpha}(\O)$ defined in \Cref{lem:Calg}. We then introduce the (closed) constraint set
\[
\mathcal B_M := \Bigl\{p\in C^{1,\alpha}(\mathcal O):\ p\ge0,\ \int_{\mathcal O}p\,d\mu=1,\ \|p\|_{C^{1,\alpha}(\mathcal O)}\le M\Bigr\}.
\]
All existence/uniqueness arguments will be carried out on an open neighborhood of $p_0$ in $C^{1,\alpha}(\mathcal O)$ and we will subsequently verify that the resulting solution remains in $\mathcal B_M$ on $[0,T]$.

\begin{lemma}[Uniform $L^{\infty}$ bound for $D_t$]
\label{lem:Dt-sup-bound}
Recall $D_t = D(p_t) := \frac1n\sum_{j=1}^n \alpha^{(t)}_j\,k^{(t)}_{O_j}$ where 
\[
\alpha^{(t)}_j := P_n\!\left[k^{(t)}_{O_j}\right]
= \frac1n\sum_{i=1}^n K^{(t)}(O_i,O_j),\qquad j=1,\dots,n.
\]
Assume the base positive definite kernel $K$ satisfies $\sup_{o\in\mathcal O} K(o,o)\le C$ for some constant $C$. Then for every $t\in I$, we have
\[
\|D_t\|_{\infty} \le C^2.
\]
In particular, the bound holds uniformly in $t$, and pathwise for any fixed sample.
\end{lemma}

\begin{proof}
Since $\mathcal H^{(t)}_K$ is a closed subspace of $\mathcal H_K$, the orthogonal projection exists. Each section satisfies $k^{(t)}_o = \Pi_t k_o$, where $\Pi_t$ denotes the orthogonal projection onto $\mathcal H^{(t)}_K$.
Therefore, by \Cref{ass::kernelbound}, diagonal of the projected kernel shrinks, 
\[
K^{(t)}(o,o)
= \|k^{(t)}_o\|^2_{\mathcal H^{(t)}_K}
= \|\Pi_t k_o\|^2_{\mathcal H_K}
\le \|k_o\|^2_{\mathcal H_K}
= K(o,o)
\le C.
\]
By Cauchy--Schwarz in $\mathcal H^{(t)}_K$,
\[
|K^{(t)}(o,o')|
= \big|\langle k^{(t)}_o,k^{(t)}_{o'}\rangle_{\mathcal H^{(t)}_K}\big|
\le \|k^{(t)}_o\|_{\mathcal H^{(t)}_K}\,\|k^{(t)}_{o'}\|_{\mathcal H^{(t)}_K}
= \sqrt{K^{(t)}(o,o)\,K^{(t)}(o',o')}
\le C.
\]
Hence, for each $j$, we can bound the coefficients,
\[
|\alpha^{(t)}_j|
=\left|\frac1n\sum_{i=1}^n K^{(t)}(O_i,O_j)\right|
\le \frac1n\sum_{i=1}^n |K^{(t)}(O_i,O_j)|
\le C.
\]
Finally, using $k^{(t)}_{O_j}(o)=K^{(t)}(o,O_j)$, we obtain for every $o\in\mathcal O$, we can bound $D_t$ pointwise as
\[
|D_t(o)|
=
\left|\frac1n\sum_{j=1}^n \alpha^{(t)}_j\,K^{(t)}(o,O_j)\right|
\le
\frac1n\sum_{j=1}^n |\alpha^{(t)}_j|\,|K^{(t)}(o,O_j)|
\le
\frac1n\sum_{j=1}^n C\cdot C
=C^2.
\]
Taking the supremum over $o\in\mathcal O$ yields $\|D_t\|_\infty\le C^2$.
\end{proof}

\Cref{lem:Dt-sup-bound} gives us a uniform $L^{\infty}$ bound for $D_t$; score-type direction $D_t$ cannot blow up pointwise as long as the kernel itself is bounded (on the diagonal). As before, for the Gaussian kernel, we have $C=1$. The next \Cref{lem:Dt-holder-bound} shows that, if kernel sections are uniformly $C^{1, \alpha}$-bounded along the flow, then $D_t$ inherits the same regularity.

\begin{lemma}
\label{lem:Dt-holder-bound}
%Assume $\mathcal O \subset \mathbb R^d$ is compact and fix $\alpha\in(0,1]$. 
Let $\O = [0,1]^d\times\{0,1\}^2$ and fix $\alpha\in(0,1]$. Suppose that for some constant $M_{1,\alpha}<\infty$, 
\begin{equation}\label{eq:kernel-section-holder-unif}
\sup_{t\in I}\ \sup_{1\le j\le n}\ \bigl\|k^{(t)}_{O_j}\bigr\|_{C^{1,\alpha}}
\le M_{1,\alpha}.
\end{equation}
Then for every $t \in I$, $D_t\in C^{1,\alpha}(\O)$ and $\sup_{t\in I}\ \|D_t\|_{C^{1,\alpha}} \le CM_{1,\alpha}$.
\end{lemma}

\begin{proof}
%Slice-case argument:
By definition, we have that $D_t=\frac1n\sum_{j=1}^n \alpha^{(t)}_j\,k^{(t)}_{O_j}$. Since $C^{1,\alpha}(\mathcal O)$ is a normed vector space, the triangle inequality yields
\[
\|D_t\|_{C^{1,\alpha}(\O)}
\le \frac1n\sum_{j=1}^n |\alpha^{(t)}_j|\, \|k^{(t)}_{O_j}\|_{C^{1,\alpha}(\O)}.
\]
By \Cref{lem:Dt-sup-bound}, we have $|\alpha^{(t)}_j|\le C$ uniformly in $j$ and $t$. Combining this with \Cref{eq:kernel-section-holder-unif} gives
\[
\|D_t\|_{C^{1,\alpha}(\O)}
\le
\frac1n\sum_{j=1}^n C\,M_{1,\alpha}
=
C\,M_{1,\alpha}.
\]
Taking the supremum over $t\in I$ completes the proof. 
\end{proof}

The next lemma establishes that $p \mapsto D(p)$ is Lipschitz continuous on $\mathcal{B}_M$ in the $C^{1, \alpha}$ norm, which is essential for proving existence and uniqueness of solutions to the density-valued ODE. 

\begin{lemma}\label{lem:D_Lipschitz}
There exists a constant $K_D<\infty$ (depending only on $C,M_K,M_{1,\alpha}$) such that for all $p,q\in \mathcal{B}_M$,
\[
\|D(p)-D(q)\|_{C^{1,\alpha}} \le K_D\,\|p-q\|_{C^{1,\alpha}}.
\]
\end{lemma}

\begin{proof}

%{\color{red} We define $k_o^{(l)}:=\Pi_l k_o$ where $\Pi_l$ is the orthogonal projection in $\mathcal H_K$ onto $\mathcal H_{K,l}={f\in\mathcal H_K:P_l=0}$. But here we use the projected section as the double-centered kernel. We should be working with $k_y^{(p)} = k_y - \frac{P_p[k_y]}{\|m_p\|^2}m_p$. Done.}

Fix $p,q\in \mathcal{B}_M$ and write $\eta:=p-q$. For any density $r\in\mathcal B_M$,
define the kernel mean embedding as before
\begin{align*}
    m_r(\cdot)&:=\int_{\mathcal O}K(\cdot,s)r(s)\,d\mu(s)\in C^{1,\alpha}(\mathcal O),\\
    \kappa_r &:= \|m_r\|_{\mathcal H_K}^2
= \iint_{\mathcal O\times\mathcal O} K(s,t)\,r(s)r(t)\,d\mu(s)d\mu(t).
\end{align*}
We work with the orthogonally projected section
\[
k_y^{(r)} := \Pi_r k_y
= k_y - \frac{P_r[k_y]}{\|m_r\|_{\mathcal H_K}^2}\,m_r
= k_y - \frac{m_r(y)}{\kappa_r}\,m_r,
\]
so that, for each $x\in\mathcal O$,
\[
k_y^{(r)}(x)=K(x,y)-\frac{m_r(y)}{\kappa_r}\,m_r(x).
\]
Assume $\kappa_0:=\inf_{r\in\mathcal B_M}\kappa_r>0$ (e.g. for the Gaussian kernel on
compact $\mathcal O$, $\kappa_0\ge \inf_{s,t\in\mathcal O}K(s,t)>0$).

Since $\mathcal O$ has finite measure and $\|\cdot\|_{C^{1,\alpha}}$ dominates $\|\cdot\|_\infty$
slice-wise, $\|\eta\|_{L^1}\le \|\eta\|_\infty\le \|\eta\|_{C^{1,\alpha}}$. Using \eqref{eq:MK_def} in \Cref{ass::kernelbound},
\begin{align}
\|m_p-m_q\|_{C^{1,\alpha}}
&=\left\|\int_{\mathcal O} K(\cdot,s)\eta(s)\,d\mu(s)\right\|_{C^{1,\alpha}}\\
&\le \int_{\mathcal O}\|K(\cdot,s)\|_{C^{1,\alpha}}|\eta(s)|\,d\mu(s) \notag\\
&\le M_K\|\eta\|_{L^1}\le M_K\|\eta\|_{C^{1,\alpha}}. \label{eq:mpmq_bound}
\end{align}
Moreover, for any fixed $y$,
\begin{equation}\label{eq:mpmq_point_bound}
|m_p(y)-m_q(y)|
=\left|\int_{\mathcal O}K(y,s)\eta(s)\,d\mu(s)\right|
\le C\|\eta\|_{L^1}\le C\|\eta\|_{C^{1,\alpha}}.
\end{equation}
Also,
\[
\kappa_p-\kappa_q
=\iint_{\mathcal O\times\mathcal O}K(s,t)\big(p(s)p(t)-q(s)q(t)\big)\,d\mu(s)d\mu(t).
\]
Note that
$p(s)p(t)-q(s)q(t)=\eta(s)p(t)+q(s)\eta(t)$, hence
\[
\big|p(s)p(t)-q(s)q(t)\big|
\le |\eta(s)|\,|p(t)| + |q(s)|\,|\eta(t)|.
\]
Since $p,q$ are densities, $\int_\mathcal{O} p\,d\mu=\int_\mathcal{O} q\,d\mu=1$ and $p,q\ge 0$; therefore,
\begin{align}
\left|\iint_{\mathcal O\times\mathcal O} K(s,t)\big(p(s)p(t)-q(s)q(t)\big)d\mu(s)d\mu(t)\right|
&\le C\iint_{\mathcal O\times\mathcal O} \big(|\eta(s)|p(t)+q(s)|\eta(t)|\big)d\mu(s)d\mu(t) \notag \\ 
&= C\left(\int_\mathcal{O} |\eta(s)|\,d\mu(s)\right)\left(\int_\mathcal{O} p(t)\,d\mu(t)\right) \notag \\
   &+ C\left(\int_\mathcal{O} q(s)d\mu(s)\right)\left(\int_\mathcal{O} |\eta(t)|d\mu(t)\right) \notag \\
&= 2C\|\eta\|_{L^1} \le 2C\|\eta\|_{C^{1,\alpha}}.
\label{eq:term3_bound}
\end{align}
Therefore,
\begin{equation}\label{eq:kappa_diff_bound}
|\kappa_p-\kappa_q|
\le 2C\|\eta\|_{L^1}\le 2C\|\eta\|_{C^{1,\alpha}}.
\end{equation}

Define $\beta_r(y):=m_r(y)/\kappa_r$. Then $|\beta_q(y)|\le C/\kappa_0$ and
\begin{align}
|\beta_p(y)-\beta_q(y)|
&\le \frac{|m_p(y)-m_q(y)|}{\kappa_0}
+|m_q(y)|\left|\frac{1}{\kappa_p}-\frac{1}{\kappa_q}\right| \notag\\
&\le \frac{C}{\kappa_0}\|\eta\|_{C^{1,\alpha}}
+\frac{C}{\kappa_0^2}|\kappa_p-\kappa_q|
\le \left(\frac{C}{\kappa_0}+\frac{2C^2}{\kappa_0^2}\right)\|\eta\|_{C^{1,\alpha}},
\label{eq:beta_diff_bound}
\end{align}
using \eqref{eq:mpmq_point_bound} and \eqref{eq:kappa_diff_bound}.

Now,
\[
k_y^{(p)}-k_y^{(q)}
= -\big(\beta_p(y)m_p-\beta_q(y)m_q\big)
= -(\beta_p(y)-\beta_q(y))m_p-\beta_q(y)(m_p-m_q).
\]
Since by \eqref{eq:MK_def}, $\|m_p\|_{C^{1,\alpha}}\le \int\|K(\cdot,s)\|_{C^{1,\alpha}}p(s)\,d\mu(s)\le M_K$,
combining \eqref{eq:mpmq_bound} and \eqref{eq:beta_diff_bound} yields
\begin{align}
\|k_y^{(p)}-k_y^{(q)}\|_{C^{1,\alpha}}
&\le |\beta_p(y)-\beta_q(y)|\,\|m_p\|_{C^{1,\alpha}}
+|\beta_q(y)|\,\|m_p-m_q\|_{C^{1,\alpha}} \notag\\
&\le M_K\left(\frac{C}{\kappa_0}+\frac{2C^2}{\kappa_0^2}\right)\|\eta\|_{C^{1,\alpha}}
+\frac{C}{\kappa_0}\,M_K\|\eta\|_{C^{1,\alpha}} \notag\\
&= M_K\left(\frac{2C}{\kappa_0}+\frac{2C^2}{\kappa_0^2}\right)\|p-q\|_{C^{1,\alpha}}.
\label{eq:k_section_lip_new}
\end{align}
Taking supremum over $y\in\mathcal O$, we obtain
\[
\sup_{y\in\mathcal O}\|k_y^{(p)}-k_y^{(q)}\|_{C^{1,\alpha}}
\le L_K \|p-q\|_{C^{1,\alpha}},
\quad
L_K:=M_K\left(\frac{2C}{\kappa_0}+\frac{2C^2}{\kappa_0^2}\right).
\]
In particular, for each $j$,
\[
\|k_{O_j}^{(p)}-k_{O_j}^{(q)}\|_{C^{1,\alpha}}
\le L_K\|p-q\|_{C^{1,\alpha}}.
\]
Recall $\alpha_j(p)=\frac1n\sum_{i=1}^n K^{(p)}(O_i,O_j)$. Therefore,
\begin{align}
|\alpha_j(p)-\alpha_j(q)|
&\le \frac1n\sum_{i=1}^n |K^{(p)}(O_i,O_j)-K^{(q)}(O_i,O_j)| \notag \\
&\le \sup_{x,y\in\mathcal{O}}|K^{(p)}(x,y)-K^{(q)}(x,y)|. \label{eq:alpha_reduce}
\end{align}
But for any fixed $y$, the quantity $|K^{(p)}(x,y)-K^{(q)}(x,y)|=|k^{(p)}_{y}(x)-k^{(q)}_{y}(x)|$
is bounded by $\|k^{(p)}_{y}-k^{(q)}_{y}\|_\infty\le \|k^{(p)}_{y}-k^{(q)}_{y}\|_{C^{1,\alpha}}$.
Hence,
\begin{equation}\label{eq:alpha_lip}
\max_{1\le j\le n}|\alpha_j(p)-\alpha_j(q)|
\le L_K\|p-q\|_{C^{1,\alpha}}.
\end{equation}
Write
\begin{align*}
D(p)-D(q)
&= \frac1n\sum_{j=1}^n\Big(\alpha_j(p)k^{(p)}_{O_j}-\alpha_j(q)k^{(q)}_{O_j}\Big) \\
&=\frac1n\sum_{j=1}^n\Big((\alpha_j(p)-\alpha_j(q))k^{(p)}_{O_j}+\alpha_j(q)(k^{(p)}_{O_j}-k^{(q)}_{O_j})\Big).    
\end{align*}
Taking $C^{1,\alpha}$ norms and using the triangle inequality,
\begin{align*}
\|D(p)-D(q)\|_{C^{1,\alpha}}
&\le \frac1n\sum_{j=1}^n\Big(
|\alpha_j(p)-\alpha_j(q)|\,\|k^{(p)}_{O_j}\|_{C^{1,\alpha}}
+|\alpha_j(q)|\,\|k^{(p)}_{O_j}-k^{(q)}_{O_j}\|_{C^{1,\alpha}}
\Big).
\end{align*}
Since $|\alpha_j(q)|\le C$ uniformly, and by \eqref{eq:M1a_def} in \Cref{ass::kernelbound},
$\|k^{(p)}_{O_j}\|_{C^{1,\alpha}}\le M_{1,\alpha}$ uniformly over $p\in \mathcal{B}_M$.
Using \eqref{eq:k_section_lip_new} and \eqref{eq:alpha_lip} , we get
\[
\|D(p)-D(q)\|_{C^{1,\alpha}}
\le (M_{1,\alpha}+C)L_K\|p-q\|_{C^{1,\alpha}}.
\]
Thus the claim holds with $K_D:=(M_{1,\alpha}+C)L_K$.
\end{proof}

\subsection{Proof of the Existence and Uniqueness of the ODE Solution}\label{append::unique}

\begin{theorem}[Existence and uniqueness of the ODE solution]
Let $p_0\in C^{1,\alpha}(\O)$ be the initial density. Assume \Cref{ass::kernelbound} holds and let $p_0\in\mathcal B_M\subset C^{1,\alpha}(\O)$. Let $F:\mathcal B_M\to C^{1,\alpha}(\O)$ be defined by $F(p):=p\,D(p)$. Then the density valued ODE in \eqref{eq::newode}, with $p_{t=0}=p_0$,
\begin{equation*}
\frac{d}{dt}p_t = F(p_t),
\end{equation*}
admits a unique solution $t\mapsto p_t\in C^1([0,T];C^{1,\alpha}(\O))$. For all $t\in[0,T]$, $p_t\in\mathcal B_M$.
\end{theorem}

\begin{proof}

First, we need to prove the solution of
$\frac{d}{dt} p_t = p_t D(p_t)$ exists and is unique in $C^1(I, C^{1,\alpha}(\mathcal{O}))$ for some interval $I=(-\tau,\tau)$, where $\tau$ is a constant. Define the operator $F$ by $F(p) = p D(p)$. The equation becomes:
$\frac{d}{dt} p_t = F(p_t),$ with initial condition $p_0$.

%Is F(p)\in C^{1,\alpha}(\O)?
%Use Banach algebra property (\Cref{lem:Calg}) and $D(p)\in C^{1,\alpha}(\O)$ for $p\in\mathcal B_M$. Lemma \ref{lem:Dt-holder-bound} gives that under the uniform kernel-section bound \eqref{eq:kernel-section-holder-unif}.
Since $p \in C^{1,\alpha}(\O)$ and $D(p) \in C^{1,\alpha}(\O)$ for $p\in\mathcal B_M$ (by the definition of $D(\cdot)$ together with \Cref{ass::kernelbound} and \Cref{lem:Dt-holder-bound}), we verify that $F(p) = p D(p) \in C^{1,\alpha}(\O)$ for $p\in\mathcal B_M$. Indeed, by the Banach algebra property of $C^{1,\alpha}(\O)$ in \Cref{lem:Calg},
\begin{equation}\label{eq:F_maps_into_C1a}
\|F(p)\|_{C^{1,\alpha}}
=\|pD(p)\|_{C^{1,\alpha}}
\le C_{\mathrm{alg}}\,\|p\|_{C^{1,\alpha}}\,\|D(p)\|_{C^{1,\alpha}},
\end{equation}
so $F(p)\in C^{1,\alpha}(\O)$.

%Local Lipschitz continuity of F
%NOTE: Picard-Lindelöf requires an open domain U on which F is locally Lipschitz.
Next, we prove local Lipschitz continuity of $F$ in the $C^{1,\alpha}$ norm. Let $\mathcal{B}_R(p_0) = \{ q \in C^{1,\alpha}(\mathcal{O}) : \|q - p_0\|_{C^{1,\alpha}} \leq R \}$ for any fixed $R>0$. Set
\begin{equation*}
U:=\{p\in C^{1,\alpha}(\O):\ \|p\|_{C^{1,\alpha}}<2M\},    
\end{equation*}
which is an open subset of $C^{1,\alpha}(\O)$ containing $\mathcal B_M$ and $p_0$. By \Cref{lem:D_Lipschitz}, the map $D(\cdot)$ is Lipschitz on $\mathcal B_M$ in the $C^{1,\alpha}$ norm with constant $K_D$, and by \Cref{ass::kernelbound} together with \eqref{eq:M1a_def} we have the uniform bound
\begin{equation}\label{eq:D_unif_bound_on_BM}
\sup_{p\in\mathcal B_M}\|D(p)\|_{C^{1,\alpha}} \le C\,M_{1,\alpha}.
\end{equation}
For $p,q \in \mathcal{B}_M \cap \mathcal{B}_R(p_0)$, we estimate
\begin{equation*}
\| F(p) - F(q) \|_{C^{1,\alpha}} = \| p D(p) - q D(q) \|_{C^{1,\alpha}}.    
\end{equation*}
Write
\begin{equation*}
F(p) - F(q) = p (D(p) - D(q)) + (p - q) D(q).    
\end{equation*}
By \Cref{lem:Calg},
\begin{equation*}
\|p(D(p)-D(q))\|_{C^{1,\alpha}}
\le C_{\mathrm{alg}}\,\|p\|_{C^{1,\alpha}}\,\|D(p)-D(q)\|_{C^{1,\alpha}},    
\end{equation*}
and
\begin{equation*}
\|(p-q)D(q)\|_{C^{1,\alpha}}
\le C_{\mathrm{alg}}\,\|p-q\|_{C^{1,\alpha}}\,\|D(q)\|_{C^{1,\alpha}}.    
\end{equation*}
Using $\|p\|_{C^{1,\alpha}}\le M$ for $p\in\mathcal B_M$, \Cref{lem:D_Lipschitz}, and \eqref{eq:D_unif_bound_on_BM}, we obtain
\begin{align}
\|F(p)-F(q)\|_{C^{1,\alpha}}
&\le C_{\mathrm{alg}}\,M\,K_D\,\|p-q\|_{C^{1,\alpha}}
   + C_{\mathrm{alg}}\,\|p-q\|_{C^{1,\alpha}}\,(C\,M_{1,\alpha}) \notag\\
&= C_{\mathrm{alg}}\bigl(MK_D + C M_{1,\alpha}\bigr)\,\|p-q\|_{C^{1,\alpha}}.\label{eq:F_Lip_on_BM}
\end{align}
Thus $F$ is Lipschitz on $\mathcal B_M$ (hence locally Lipschitz on the open set $U$), with Lipschitz constant $L_F:=C_{\mathrm{alg}}\bigl(MK_D + C M_{1,\alpha}\bigr)$.

%Applying Picard-Lindelöf to get a local solution
By the Picard-Lindelöf theorem in Banach spaces (\Cref{thm::picard}) applied on the open set $U$, there exists $\tau>0$ and a unique solution $p\in C^1((-\tau,\tau);C^{1,\alpha}(\O))$  to $\frac{d}{dt}p_t = F(p_t)$ with $p_{t=0}=p_0$. The solution satisfies the integral equation
\begin{equation}\label{eq:integral_form_picard}
p_t=p_0+\int_0^t p_sD(p_s)\,ds,\qquad t\in(-\tau,\tau).
\end{equation}
Since $p_t\in C^{1,\alpha}(\O)$ for each $t$ and $C^{1,\alpha}(\O)$ embeds continuously into $C(\O)$, the ODE holds pointwise on $\O$.

%Mass conservation: $\int p_t,d\mu$ stays equal to 1 and positivity
Next, we confirm $p_t \in \mathcal{B}_M$ on $[0,T]$. First, consider the continuous linear functional \(\Phi: C^{1,\alpha}(\mathcal{O}) \to \mathbb{R}\) defined by \(\Phi(q) = \int_\mathcal{O} q(o) \, d\mu(o)\). For any small $\delta \neq 0$, by linearity and continuity of $\Phi$,
\begin{equation*}
\frac{\Phi(p(t+\delta))-\Phi(p(t))}{\delta}
=\Phi\!\left(\frac{p(t+\delta)-p(t)}{\delta}\right)
\xrightarrow[\delta\to 0]{\|\cdot\|_{C^{1,\alpha}}}  \Phi(p^{'}(t))  
\end{equation*}
Therefore,
\begin{equation*}
\frac{d}{dt} \int_{\mathcal{O}} p_t(o) \, d\mu(o)
= \int_{\mathcal{O}} \frac{d}{dt} p_t(o) \, d\mu(o)
= \int_{\mathcal{O}} p_t(o) D(p_t)(o) \, d\mu(o)
= P_t[D(p_t)].    
\end{equation*}
Under the orthogonal projection definition, $D(p_t)\in\mathcal H_{K,p_t}:=\{f\in\mathcal H_K:P_t[f]=0\}$ for every $t$ by construction, hence $P_t[D(p_t)]=0$. Since $\int p_0(o) \, d\mu(o) = 1$, we conclude $\int p_t(o) \, d\mu(o) = 1$ for all $t\in(-\tau,\tau)$. Positivity holds because for each fixed $o\in\O$, the pointwise ODE $\frac{d}{dt}p_t(o)=p_t(o)D(p_t)(o)$ implies
\begin{equation}\label{eq:explicit_solution_pointwise}
p_t(o) = p_0(o)\exp\!\left(\int_0^t D(p_s)(o)\,ds\right),
\qquad t\in(-\tau,\tau),
\end{equation}
so $p_t(o)\ge 0$ whenever $p_0(o)\ge 0$.

%$C^{1,\alpha}$-norm bound via Grönwall
Next, we check that $\|p_t\|_{C^{1,\alpha}}\leq M$ for $t\in[0,T]$. From \eqref{eq:integral_form_picard} and \Cref{lem:Calg},
\begin{align*}
\|p_t\|_{C^{1,\alpha}}
&\le \|p_0\|_{C^{1,\alpha}}+\int_0^t\|p_sD(p_s)\|_{C^{1,\alpha}} \,ds\\
&\le \|p_0\|_{C^{1,\alpha}}
+C_{\mathrm{alg}}\int_0^t\|p_s\|_{C^{1,\alpha}}\|D(p_s)\|_{C^{1,\alpha}} \,ds\\
&\le \|p_0\|_{C^{1,\alpha}}
+C_{\mathrm{alg}}\,C\,M_{1,\alpha}\int_0^t\|p_s\|_{C^{1,\alpha}} \,ds,
\end{align*}
where in the last line we used \eqref{eq:D_unif_bound_on_BM} (valid as long as $p_s\in\mathcal B_M$). Grönwall's inequality yields
\begin{equation}\label{eq:gronwall_bound_pt}
\|p_t\|_{C^{1,\alpha}}
\le \|p_0\|_{C^{1,\alpha}}\exp\!\bigl(C_{\mathrm{alg}}\,C\,M_{1,\alpha}\,t\bigr),
\qquad t\in[0,\tau).
\end{equation}
By the definition of $M$ and $T$ given above (namely $M:=M_0\exp(C_{\mathrm{alg}}CM_{1,\alpha}T)$ with $M_0:=\|p_0\|_{C^{1,\alpha}}$, and $T=\frac{2}{\delta_n}\log(2\mathfrak J(\delta_n))$), we have for all $t\in[0,T]$ that
\begin{equation*}
\|p_t\|_{C^{1,\alpha}}
\le M_0\exp(C_{\mathrm{alg}}CM_{1,\alpha}t)
\le M_0\exp(C_{\mathrm{alg}}CM_{1,\alpha}T)
= M,    
\end{equation*}
so $\|p_t\|_{C^{1,\alpha}}\le M$ on $[0,T]$.

Finally, since $F$ is locally Lipschitz on the open set $U$ and the bound \eqref{eq:gronwall_bound_pt} shows that the solution cannot exit $U$ before time $T$ (indeed $\|p_t\|_{C^{1,\alpha}}\le M<2M$ on $[0,T]$), we can extend the local solution uniquely to $[0,T]$ by the standard continuation argument for ODEs in Banach spaces (restarting the Picard--Lindelöf theorem at intermediate times). Thus there exists a unique solution
\[
p\in C^1([0,T];C^{1,\alpha}(\O))
\]
and for all $t\in[0,T]$ we have $p_t\ge 0$, $\int_\O p_t\,d\mu=1$, and $\|p_t\|_{C^{1,\alpha}}\le M$, i.e., $p_t\in\mathcal B_M$.

Until now, we have proven that for all $t\in[0,T]$, $\frac{d}{dt}p_t(o)=p_t(o)D(p_t)(o)$ holds pointwise for all $o\in \mathcal{O}$, and $p_t\in\mathcal{B}_M$. Notice that this solves \eqref{eq::newode}. Moreover, on the set $\{o\in\O:p_t(o)>0\}$, it also implies $\frac{d}{dt}\log p_t(o)=\frac{\frac{d}{dt}p_t(o)}{p_t(o)}$.

\end{proof}

\subsection{Proof of the Convergence of the Empirical Score}\label{append::convergence}

\begin{lemma}
\label{lemma:Ecnorm}
We have 
$$\sup_{p\in \mathcal{B}_M}\sup_{o\in \mathcal{O}}\|\triangledown D(p)(o)\|_{2}\leq CM_{1,\alpha}<\infty,$$
where $C, M_{1,\alpha}$ are defined in \Cref{ass::kernelbound}.
\end{lemma}

\begin{proof}
Fix any slice $b\in\{0,1\}^2$ and write $v(x):=D(p)(x,b)$ on $[0,1]^d$.
By definition of the $C^{1,\alpha}$ norm on a slice,
\[
\|v\|_{C^{1,\alpha}}
=\|v\|_\infty + \sum_{i=1}^d \|\nabla_i v\|_\infty + \sum_{i=1}^d [\nabla_i v]_\alpha.
\]
Hence, since all terms are nonnegative, we have that
\[
\sum_{i=1}^d \|\nabla_i v\|_\infty \;\le\; \|v\|_{C^{1,\alpha}}.
\]
For any $x\in[0,1]^d$,
\[
\|\nabla_x v(x)\|_2
=\Big(\sum_{i=1}^d |\nabla_i v(x)|^2\Big)^{1/2}
\le \Big(\sum_{i=1}^d \|\nabla_i v\|_\infty^2\Big)^{1/2}
\le \sum_{i=1}^d \|\nabla_i v\|_\infty
\le \|v\|_{C^{1,\alpha}}.
\]
Taking the maximum over slices and using
\[
\|D(p)\|_{C^{1,\alpha}(\mathcal{O})}=\max_{b}\ \|D(p)(\cdot,b)\|_{C^{1,\alpha}([0,1]^d)},
\]
we obtain, for any $o=(x,b)\in\mathcal{O}$ and $p\in\mathcal{B}_M$,
\[
\|\nabla_x D(p)(o)\|_2
\le \|D(p)(\cdot,b)\|_{C^{1,\alpha}([0,1]^d)}
\le \|D(p)\|_{C^{1,\alpha}(\mathcal{O})}.
\]
Finally, taking the supremum over $o\in\mathcal{O}$ and $p\in\mathcal B_M$ gives
\[
\sup_{p\in\mathcal B_M}\ \sup_{o\in\mathcal{O}}\ \|\nabla_x D(p)(o)\|_2
\le \sup_{p\in\mathcal B_M}\, \|D(p)\|_{C^{1,\alpha}(\mathcal{O})}
\le CM_{1,\alpha}<\infty,
\]
by \Cref{lem:Dt-holder-bound}.
\end{proof}

\Cref{lemma:Ecnorm} proves a useful consequence of the uniform norm $C^{1,\alpha}$ control of $D(p)$ on $\mathcal{B}_M$. It yields a uniform bound on the Euclidean norm of the gradient of $D(p)$ with respect to the continuous covariate coordinate, which will be used to control how $D(p)$ varies on small neighborhoods around observed sample points. 

\begin{lemma}[Monotonicity and stationarity of the empirical log-likelihood]
Let $t \mapsto p_t$ be a sufficiently regular solution to the differential equation 
\begin{equation*}
\frac{d}{dt}\log p_t(o)=D(p_t)(o).
\end{equation*}
with $D(p_t) := \frac1n \sum_{j=1}^n \alpha_j^{(t)}\kt_{O_j}$. Then the following hold:
\begin{enumerate}
\item The map $t \mapsto \P_n[\log p_t]$ is nondecreasing on $I$.
\item For any $t_1 \in I$,
\begin{equation*}
\left.\frac{d}{dt}\P_n[\log p_t]\right|_{t=t_1} = 0
\quad\Longleftrightarrow\quad
m_n^{(t_1)} = 0 \ \text{in } \HK^{(t_1)}.
\end{equation*}
\end{enumerate}
\end{lemma}

\begin{proof}
Differentiating $\P_n[\log p_t]$ along the flow and using definition of $D(p_t)$ with \eqref{eq::ode} yields
\begin{align*}
\frac{d}{dt}\P_n[\log p_t]
= \P_n[D(p_t)]
&= \frac{1}{n}\sum_{i=1}^n D(p_t)(O_i)  \\
&= \frac{1}{n}\sum_{i=1}^n \frac1n\sum_{j=1}^n \alpha_j^{(t)}\,\Kt(O_j,O_i) \nonumber \\
&= \frac1n\sum_{j=1}^n \alpha_j^{(t)}\,\P_n[\kt_{O_j}] \nonumber \\
&= \frac1n\sum_{j=1}^n \bigl(\alpha_j^{(t)}\bigr)^2
\ge0. \nonumber
\end{align*}
This establishes monotonicity. If the derivative vanishes at some $t_1 \in I$, then $\alpha_j^{(t_1)} = 0$ for all $j=1,\ldots,n$. Since $\alpha_j^{(t_1)}
= \bigl\langle \kt_{O_j}, m_n^{(t_1)} \bigr\rangle_{\HK^{(t_1)}}$, this implies that $m_n^{(t_1)}$ is orthogonal to $\operatorname{span}\{\kt_{O_1},\ldots,\kt_{O_n}\}$.
However, by definition, $m_n^{(t_1)} = \frac1n\sum_{i=1}^n \kt_{O_i}
\in \operatorname{span}\{\kt_{O_1},\ldots,\kt_{O_n}\}$, and therefore $m_n^{(t_1)}=0$ in $\HK^{(t_1)}$. Conversely, if $m_n^{(t_1)} = 0$, then all $\alpha_j^{(t_1)}=0$, and the derivative
$\frac{d}{dt}\P_n[\log p_t]$ vanishes at $t_1$.
\end{proof}

\begin{theorem}
There exists $t \in [0 ,T]$ such that $\P_nD(p_t) \leq \delta_n$.
\end{theorem}

\begin{proof}
By \Cref{t::picard}, the solution $p_t$ exists and is unique on $[0,T]$. Suppose, for the sake of obtaining a contradiction, that $\P_nD(p_t) > \delta_n$ for all $t \in [0,T]$. Then integrating (\ref{eq::newode}) from $0$ to $T$, since $P_n\frac{d}{dt}\log p_t=\P_nD(p_t)$, we find that 
\[
P_n [ \log p_T] - P_n [ \log p_0] > \delta_n T.
\]
Using the representation
\begin{equation}\label{e:exprepT}
p_T = p_{0} \exp \left( \int_{0}^T D(p_s) \, ds \right),
\end{equation}
this gives 
\[
\frac{1}{n} \sum_{i=1}^n \int_0^T D(p_s)(O_i) \, ds > \delta_n T.
\]
Then, there must exist an index $k$ such that $\int_0^T D(p_s)(O_k) \, ds > \delta_n T$. By \Cref{lemma:Ecnorm}, $D$ is uniformly bounded in Euclidean norm, and we have for all $o \in \mathcal{O}_k$ that 
\[
D({p_s})(o) \ge D^{*}(p_s)(O_k)  - CM_{1,\alpha} \|o - O_k\|_2.
\]
Therefore, 
\[
\int_0^T D({p_s})(o)\, ds  \ge \int_0^T D({p_s})(O_k) \, ds  - CM_{1,\alpha} T \|o - O_k\|_2.
\]
By \eqref{e:exprepT}, this implies that 
\begin{align*}
\begin{split}
p_T(o) &\ge p_0(o) \exp\left( \int_0^T D({p_s})(O_k) \, ds  - CM_{1,\alpha} T \|o - O_k\|_2 \right)\\
&> p_0(o) \exp\left(  \delta_n T  -   CM_{1,\alpha} T \|o - O_k\|_2  \right) .
\end{split}
\end{align*}
Finally, we have that
\begin{align*}
\begin{split}
\int_{\mathcal{O}_{A_k,Y_k}} p_T(o) \, d\mu(o) & \ge
\int_{B(O_k, w)} p_T(o) \, d\mu(o)  \\
&> 
\int_{B(O_k, w)} p_0(o) \exp\left( \frac{\delta_n T}{2} \right) \, d\mu(o) \\
& = 
\exp\left( \frac{\delta_n T}{2} \right) \int_{B(O_k, w)} p_0(o)  \, d\mu(o).
\end{split}
\end{align*}
Since $p_T$ is a probability density, we reach a contradiction when 
\[
\exp\left( \frac{\delta_n T}{2} \right)
> \left(\int_{B(O_k, w(\delta_n))} p_0(o)  \, d\mu(o)\right)^{-1}.
\]
This completes the proof.
\end{proof}

\subsection{Asymptotic Linearity and Efficiency}
\label{append::efficiency}

\begin{theorem}
Let $\Psi: \mathcal{M}\to \R$ be a pathwise-differentiable functional of the distribution $P$ with canonical gradient $\phi^*_{P}\in L_0^2(P)$. Let $t_n\in [0,T]$, which satisfies $\P_nD(p_{t_n})\leq \delta_n$, and $\hat P_n:=P_{t_n}$. Under Assumptions \ref{ass::emprocess}, \ref{ass::2orderrem}, \ref{ass::kernelbound}, \ref{ass::universal}, $\P_n\phi^*_{P_{t_n}}=o_{P^*}(n^{-1/2})$ and the ULFS--KDPE estimator satisfies 
\begin{equation*}
\Psi(\hat P_n)-\Psi(P^*)=\P_n\phi^*_{P^*}+o_{P^*}(n^{-1/2}).
\end{equation*}
\end{theorem}

\begin{proof}
By the von Mises expansion (\Cref{prop::vmHatP}), we have
\begin{equation}\label{eq:vonmises_decomp_ulfp}
\Psi(\hat P_n)-\Psi(P^*) = \P_n\phi^*_{P^*} - \P_n\phi^*_{\hat P_n} + (\P_n-P^*)(\phi^*_{\hat P_n}-\phi^*_{P^*}) + R_2(\hat P_n,P^*).
\end{equation}
Under Assumptions~\ref{ass::emprocess} and \ref{ass::2orderrem}, it suffices to show
$\P_n\phi^*_{\hat P_n}=o_{P^*}(n^{-1/2})$.

Fix $t\in[0,T]$ and let $m_n^{(t)}\in\mathcal H_{K,p_t}$ be the empirical mean embedding, i.e.
$P_n[f]=\langle f,m_n^{(t)}\rangle_{\mathcal H_K^{(t)}}$ for all $f\in\mathcal H_K^{(t)}$. Recall $\alpha_j^{(t)}:=\P_n[k^{(t)}_{O_j}]$ and $\P_nD(p_t)=\frac1n\sum_{j=1}^n(\alpha_j^{(t)})^2 \ge0$. Taking $f=m_n^{(t)}$ yields
\[
\P_n[m_n^{(t)}] = \langle m_n^{(t)}, m_n^{(t)}\rangle_{\mathcal H_K^{(t)}}
= \|m_n^{(t)}\|_{\mathcal H_K^{(t)}}^2.
\]
Since $\alpha_j^{(t)}=m_n^{(t)}(O_j)$, we have
\[
\P_n[m_n^{(t)}] = \frac1n\sum_{j=1}^n m_n^{(t)}(O_j)
= \frac1n\sum_{j=1}^n \alpha_j^{(t)}.
\]
By Cauchy--Schwarz,
\[
\Big(\frac1n\sum_{j=1}^n \alpha_j^{(t)}\Big)^2
\le
\frac1n\sum_{j=1}^n (\alpha_j^{(t)})^2.
\]
Finally, by the ULFP identity $\P_nD(p_t)=\frac1n\sum_{j=1}^n (\alpha_j^{(t)})^2$, we obtain
\[
\|m_n^{(t)}\|_{\mathcal H_K^{(t)}}^4
=
\Big(\P_n[m_n^{(t)}]\Big)^2
=
\Big(\frac1n\sum_{j=1}^n \alpha_j^{(t)}\Big)^2
\le
\P_nD(p_t),
\]
which implies $\|m_n^{(t)}\|_{\mathcal H_K^{(t)}} \le \{\P_nD(p_t)\}^{1/4}$.
We know $\P_nD(p_{t_n})\le \delta_n$ and $\delta_n=o(n^{-2})$, hence we obtain
\begin{equation}\label{eq:mn_rate}
\|m_n^{(t_n)}\|_{\mathcal H_K^{(t_n)}} \le \delta_n^{1/4} = o(n^{-1/2}).
\end{equation}

By Assumption~\ref{ass::universal}(i), for each sample path $\omega$ we can choose an index
$j_n(\omega)\uparrow\infty$ such that
\begin{equation}\label{eq:hj_approx_rate}
\|h_{j_n}-\phi^*_{\hat P_n}\|_{L^2(P^*)}=o_{P^*}(n^{-1/2}).
\end{equation}
Decompose
\begin{equation}\label{eq:decomp_plug_in}
\P_n\phi^*_{\hat P_n}
=
\P_n h_{j_n}
+
\P_n(\phi^*_{\hat P_n}-h_{j_n}).
\end{equation}
Since $h_{j_n}\in\mathcal H_K^{(t_n)}$, the Riesz identity yields
$\P_n h_{j_n}=\langle h_{j_n},m_n^{(t_n)}\rangle_{\mathcal H_K^{(t_n)}}$, hence
\[
|\P_n h_{j_n}|
\le
\|h_{j_n}\|_{\mathcal H_K^{(t_n)}}\,
\|m_n^{(t_n)}\|_{\mathcal H_K^{(t_n)}}.
\]
By Assumption~\ref{ass::universal}(ii) and \eqref{eq:mn_rate}, it follows that
\begin{equation}\label{eq:termA_small}
\P_n h_{j_n}=o_{P^*}(n^{-1/2}).
\end{equation}
Next, we write
\[
\P_n(\phi^*_{\hat P_n}-h_{j_n})
=
(\P_n-P^*)(\phi^*_{\hat P_n}-h_{j_n})
+
P^*(\phi^*_{\hat P_n}-h_{j_n}).
\]
By Cauchy--Schwarz and \eqref{eq:hj_approx_rate},
\begin{equation}\label{eq:termB_pop}
|P^*(\phi^*_{\hat P_n}-h_{j_n})|
\le
\|\phi^*_{\hat P_n}-h_{j_n}\|_{L^2(P^*)}
=
o_{P^*}(n^{-1/2}).
\end{equation}
Moreover, by Assumption~\ref{ass::universal}(iii), the tail class is $P^*$-Donsker.
Since $\|\phi^*_{\hat P_n}-h_{j_n}\|_{L^2(P^*)}\to0$, asymptotic equicontinuity
of the empirical process implies
\begin{equation}\label{eq:termB_emp}
(\P_n-P^*)(\phi^*_{\hat P_n}-h_{j_n})
=
o_{P^*}(n^{-1/2}).
\end{equation}
Combining \eqref{eq:termB_pop} and \eqref{eq:termB_emp} gives
\begin{equation}\label{eq:termB_small}
\P_n(\phi^*_{\hat P_n}-h_{j_n})=o_{P^*}(n^{-1/2}).
\end{equation}
Finally, \eqref{eq:decomp_plug_in} together with \eqref{eq:termA_small} and
\eqref{eq:termB_small} yields
\begin{equation}\label{eq:plug_in_bias_final}
\P_n\phi^*_{\hat P_n}=o_{P^*}(n^{-1/2}).
\end{equation}
Substituting \eqref{eq:plug_in_bias_final} into \eqref{eq:vonmises_decomp_ulfp},
and using Assumptions~\ref{ass::emprocess} and \ref{ass::2orderrem}, we obtain
\[
\Psi(\hat P_n)-\Psi(P^*)
=
\P_n\phi^*_{P^*}
+
o(n^{-1/2}),
\]
which completes the proof.
\end{proof}

%%%%%%%
%%%%%%%
\bibliographystyle{plain}
\bibliography{one_step_kdpe.bib}

\end{document}